\documentclass[10pt]{article} \usepackage[total={155mm,235mm}]{geometry}

\usepackage{tikz}
\usetikzlibrary{cd}
\usepackage[all,cmtip]{xy}
\usepackage{amsmath,amsthm,amssymb}
\usepackage{mathtools,bbm}
\usepackage{color,enumitem}
\usepackage{verbatim} 
\usepackage{hyperref} 
\input diagxy

\DeclareMathOperator{\cl}{cl}
\DeclareMathOperator{\nb}{int}

\DeclareMathOperator{\sub}{sub}

\newcommand{\bbB}{\mathbb{B}}

\newcommand{\bbS}{\mathbb{S}}
\newcommand{\bbL}{\mathbb{L}}
\newcommand{\CT}{\mathcal{T}}
\newcommand{\CO}{\mathcal{O}}

\newcommand{\CF}{\mathcal{F}}
\newcommand{\CG}{\mathcal{G}}
\newcommand{\Btop}{\mathbb{B}\text{-}{\bf Top}}
\newcommand{\dFrm}{\mathbf{dFrm}}
\newcommand{\BFrm}{\mathbb{B}\downarrow\mathbf{Frm}}
\newcommand{\lam}{\lambda}
\newcommand{\dbt}{{t\!t}}
\newcommand{\dbf}{{f\!\!f}}
\newcommand{\ra}{\rightarrow}
\newcommand{\lra}{\longrightarrow}
\newcommand{\id}{\rm id}
\newcommand{\Bpt}{\bbB{\rm pt}}
\newcommand{\BO}{\bbB\mathcal{O}}    
\newcommand{\dO}{{\rm d}\mathcal{O}}
\newcommand{\bv}{\bigvee}
\newcommand{\bw}{\bigwedge}

\theoremstyle{plain}
\newtheorem{thm}{Theorem}[section]
\newtheorem{lem}[thm]{Lemma}
\newtheorem{prop}[thm]{Proposition}
\newtheorem{cor}[thm]{Corollary}
\theoremstyle{definition}
\newtheorem{defn}[thm]{Definition}

\newtheorem{exmp}[thm]{Example}

\allowdisplaybreaks 

\title{ A Boolean-valued space approach to separation axioms\\ and sobriety of bitopological spaces} \author{Jing He, Dexue Zhang\\ {\normalsize School of Mathematics, Sichuan University, Chengdu, China}\\ {\normalsize Email: 2267989324@qq.com, dxzhang@scu.edu.cn}}  
\date{}  
\begin{document}
	\maketitle
		
	\begin{abstract} This paper presents a study of separation axioms and sobriety of bitopological spaces from the point of view of fuzzy topology via identifying   bitopological spaces with   topological spaces valued in the Boolean algebra of four elements. A system of separation axioms is proposed making use of Boolean-valued specialization order of bitopological spaces; The relationship between d-sobriety of bitopological spaces proposed by Jung and Moshier and  sobriety of fuzzy topological spaces is studied; A Hofmann-Mislove theorem for bitopological spaces is established. 
		 \end{abstract}

		\section{Introduction} 
		
		Bitopological spaces were introduced in 1963 by Kelly \cite{Kelly}. Since then, these spaces have attracted attention in different areas because of their intrinsic connection to quasi-metric spaces, quasi-uniform spaces \cite{Kelly,Lane,Patty}, stably compact spaces \cite{Lawson2010}, and Stone dualities \cite{Jakl,JM2006,JM2008}. In this paper, we study these spaces from the point of view  of fuzzy topology: they are a very special kind of fuzzy topological spaces, namely, fuzzy topological spaces valued in the Boolean algebra $\bbB=\{1,0,\dbt,\dbf\}$ of four elements. A benefit of doing so is that we can apply ideas and results in theories of fuzzy ordered sets (or quantale-enriched categories) and fuzzy topological spaces to the investigation of bitopological spaces. We demonstrate this by proposing a new system of separation axioms and establishing a Hofmann-Mislove theorem for bitopological spaces.  
		
		The contents are arranged as follows. 
		In Section \ref{first section} we recall some basic notions of bitopological spaces and $\bbB$-valued topological spaces, and show that the corresponding categories are isomorphic to each other.
		
		In Section \ref{second section}  a system of separation axioms for bitopological spaces is proposed, including $R_0,R_1$, $T_0,T_1$, Hausdorff, regular, and normal. A key feature of the postulations is that the $\bbB$-valued specialization order of a bitopological space plays a central role. In a recent work of Arrieta, Guti\'{e}rrez Garc\'{i}a and  H\"{o}hle, specialization fuzzy order has been used to characterize the axioms $T_0$ and   $T_1$ for fuzzy topological spaces, see \cite[page 13]{AGH}. But, even in the $\bbB$-valued setting, the $T_1$ axiom here is quite different from that in \cite{AGH}. The postulations presented here may also be applied to fuzzy topological spaces valued in an arbitrary unital quantale, resulting in a system of separation axioms for quantale-valued topological spaces that are different from those in the literature, e.g. \cite{AGH,HR80,Kubiak,Srivastava}.
		
		In Section \ref{third section} we focus on connections between bitopological spaces and their algebraic duals, which are analogous to that between topological spaces and frames. Specifically, we study the relationship between d-sobriety proposed in \cite{JM2006,JM2008} for bitopological spaces and $\bbB$-sobriety proposed in \cite{Zhang2018,ZL95} for fuzzy topological spaces, and establish a Hofmann-Mislove theorem for bitopological spaces that is parallel to that for topological spaces.
		
		We refer to \cite{AHS} for category theory, to \cite{Willard} for general topology, and to  \cite{Gierz2003,Johnstone,PP2012} for ordered sets, frames and sober spaces.
		
		\section{\texorpdfstring{Bitopological spaces and $\mathbb{B}$-topological spaces}{}}  \label{first section}
		
		A bitopological space \cite{Kelly} is a triple $(X, \tau[\dbt], \tau[\dbf])$, where $X$ is a set, $\tau[\dbt]$ and $\tau[\dbf]$ are  topologies on $X$. A continuous (also called bicontinuous in the literature) map $$f\colon(X, \tau_X[\dbt], \tau_X[\dbf]) \lra (Y,\tau_Y[\dbt], \tau_Y[\dbf])$$  is a function $f\colon X\lra Y$ such that both $f\colon(X,\tau_X[\dbt])\lra (Y,\tau_Y[\dbt])$ and  $f\colon(X,\tau_X[\dbf])\lra (Y,\tau_Y[\dbf])$ are continuous. The category of bitopological spaces and  continuous maps is denoted by $$\bf{BiTop}.$$

		Let $\mathbb{B}$ denote the  Boolean algebra $\{ 0,1,\dbt, \dbf \}$,  with $0$ being the bottom element, $1$ being the top element,   $\dbt$ and $\dbf$ being complement of each other, 
		as visualized below:
		\[\bfig \morphism(0,0)/@{-}/<-250,-250>[1`\dbt;]  \morphism(0,0)/@{-}/<250,-250>[1`\dbf;]
		\morphism(-250,-250)/@{-}/<250,-250>[\dbt`0;]  \morphism(250,-250)/@{-}/<-250,-250>[\dbf`0;]
		\efig\]
		As we shall see, the category of bitopological spaces is isomorphic to the category of $\bbB$-valued topological spaces. We introduce some notations first. \begin{itemize}\item The  arrow $\ra$  denotes the implication operator of the Boolean algebra $\bbB$. Explicitly,  $a\ra b=\neg a\vee b$  for all $a,b\in\bbB$, where $\neg$ is the negation operator on $\bbB$; that is, $$\neg\dbt=\dbf,~\neg\dbf=\dbt,~\neg0=1,~\neg1=0.$$  \item For each set $X$, $\bbB^X$ denotes the set of functions $X\lra\bbB$. Following the practice in fuzzy set theory, each element of $\bbB^X$ is called a \emph{$\bbB$-valued subset} of $X$. It is clear that $\bbB^X$ is a complete lattice when ordered pointwise, indeed it is a complete Boolean algebra. For each $\lam\in\bbB^X$, each $A\subseteq X$ and each $b\in\bbB$, the (usual) subset $\lam[b]$ and the $\bbB$-valued subset $b_A$ of $X$  are defined by  $$\lam[b]=\{x\in X: \lam(x)\geq b\}, \quad b_A(x)=\begin{cases} b & x\in A,\\ 0 &{\rm otherwise}.\end{cases}  $$ In particular, for each $x\in X$,  the function $1_x\colon X\lra\bbB$  sends $x$ to $1$ and other points to $0$. \item For each $\lam\in\bbB^X$  and each $b\in\bbB$, the $\bbB$-valued subsets  $b\wedge\lam$ and  $b\ra\lam$ of $X$ are defined by  $$  (b\wedge\lam)(x)=  b\wedge\lam(x),\quad (b\ra\lam)(x)=  b\ra\lam(x).$$ 
		\end{itemize}

		Suppose $X$ is a set. A \emph{$\mathbb{B}$-valued topology} (a $\mathbb{B}$-topology for short) on   $X$ is a subset $\tau$ of $\mathbb{B}^X$ subject to the following conditions:
		\begin{enumerate}[label=\rm(\roman*)]\setlength{\itemsep}{0pt} \item Each constant function $X\lra\mathbb{B}$ is in $\tau$.
			\item The (pointwise) join   of   any subset of $\tau$ is in $\tau $.
			\item The (pointwise) meet of  any finite subset of $\tau$ is in $\tau $.
		\end{enumerate}
		
		A set $X$ together with a $\bbB$-topology $\tau$   is called a \emph{$\mathbb{B}$-topological space}, elements of $\tau$ are called \emph{open  sets}. $\mathbb{B}$-topological spaces are a special kind of quantale-valued (or, fuzzy) topological spaces, see e.g. \cite{Kubiak,Lowen76}.  
		
		A map $f\colon(X,\tau_X)\lra (Y,\tau_Y)$ between $\mathbb{B}$-topological spaces is  \emph{continuous} if $\lambda\circ f\in\tau_X$  for all $\lambda\in\tau_Y$. 
		The category of $\mathbb{B}$-topological spaces and   continuous maps is denoted by  $\mathbb{B}\text{-}{\bf Top}.$

		For each $\bbB$-topological space $(X,\tau)$, let $$\tau[\dbt]=\{\lam[\dbt]: \lam\in\tau\}; ~~ \tau[\dbf]=\{\lam[\dbf]: \lam\in\tau\}.$$ Then $(X, \tau[\dbt],\tau[\dbf])$ is a bitopological space. 
		For each bitopological space $(X,\tau[\dbt], \tau[\dbf])$, let \[\tau= \{ \lambda\in\mathbb{B}^X: \lambda[\dbt] \in\tau[\dbt],\lambda[\dbf]\in\tau[\dbf] \}.\] Then $(X,\tau)$ is a $\bbB$-topological space.  These correspondences are  functorial and  inverse to each other, so ${\bf BiTop}$ and $\Btop$ are isomorphic to each other. Thus, we do not distinguish  a $\bbB$-topological space $(X,\tau)$ from the bitopological space $(X,\tau[\dbt], \tau[\dbf])$,  we'll switch between them freely. 
		
		\begin{exmp} \label{Sierpinski} Let $\tau_\mathbb{S}$ be the $\bbB$-topology on $\bbB$ having  $\{\id_\bbB\}$ as  a subbasis, where $\id_\bbB$ is the identity map on $\bbB$. For each $\bbB$-topological space $(X,\tau)$ and each $\lam\in\bbB^X$, it is clear that $\lam\in\tau$ if and only if $\lam\colon(X,\tau) \lra(\bbB,\tau_\bbS)$ is continuous. Thus, in the category $\Btop$ the space $(\bbB,\tau_\bbS)$ plays a role analogous to the Sierpi\'{n}ski space in the category of topological spaces. In terms of bitopological spaces, the space $(\bbB,\tau_\bbS)$ is  $(\bbB,\tau_\bbS[\dbt],\tau_\bbS[\dbf])$, where $\tau_\bbS[\dbt]=\{\emptyset, \{\dbt,1\},\bbB\}$ and $\tau_\bbS[\dbf]=\{\emptyset, \{\dbf,1\},\bbB\}$. \end{exmp}

		\begin{exmp} \label{X.Y}  For any topological spaces $X$ and $Y$, Jung and Moshier \cite[page 38]{JM2006} constructed a bitopological space $X.Y$ as follows: $$ X.Y=(X\times Y,\tau[\dbt], \tau[\dbf]),$$ where \begin{align*}  \tau[\dbt]&=\{U\times Y: U~\text{is an open set of $X$}\},\\ \tau[\dbf]&=\{X\times V: V~\text{is an open  set of $Y$}\}.\end{align*} The following facts about $X.Y$ will be used in this paper. \begin{enumerate}[label=\rm(\roman*)]\setlength{\itemsep}{0pt} \item Each nonempty open set of $(X\times Y,\tau[\dbt])$ intersects each nonempty open set of $(X\times Y,\tau[\dbf])$. \item Each nonempty closed set of $(X\times Y,\tau[\dbt])$ intersects each nonempty closed set of $(X\times Y,\tau[\dbf])$. \item A function $\lambda\colon X\times Y\lra\bbB$ is a member of $\tau$ if and only if there exist an open set $U$ of $X$ and an open set $V$ of $Y$ such that $$\lambda(x,y)=\begin{cases} 1 & x\in U,~y\in V, \\ \dbt & x\in U,~y\notin V, \\ \dbf & x\notin U,~y\in V,\\ 0 & x\notin U,~y\notin V.  \end{cases} $$   \end{enumerate} \end{exmp}
		
		We need some terminology from fuzzy topology. 
		Suppose $(X,\tau)$ is a $\bbB$-topological space and  $\lam\in\bbB^X$.  \begin{enumerate}[label=\rm(\alph*)]\setlength{\itemsep}{0pt} \item  We say that $\lambda$ is a \emph{closed  set}   if $\neg\lambda$ is an open set, where $\neg$ is the negation operator on  $\bbB$ extended pointwise to $\bbB^X$.  Since $\bbB$ is a Boolean algebra, it is readily seen that $\lam$ is a closed set of $(X,\tau)$ if and only if $\lam[\dbt]$ is a closed set of $(X,\tau[\dbt])$ and $\lam[\dbf]$ is a closed set of $(X,\tau[\dbf])$.
			\item  The \emph{interior} $\lam^\circ$ of $\lambda$  is defined to be the open set  $\bigvee \{\mu\in\tau: \mu\leq\lambda \}$.
			\item The \emph{closure} $\overline{\lambda}$ of $\lambda$ is defined to be the closed set $\bigwedge \{\mu:\neg\mu\in\tau, \mu\geq\lambda\}.$
		\end{enumerate}
		
		We list here some basic properties of the closure operator and the interior operator of a $\mathbb{B}$-topological space $(X,\tau)$, leaving verification   to the reader. 
		
		For all $\lam,\mu\in\mathbb{B}^X$ and all $b\in\bbB$,  (i) $\lam\leq \overline{\lam}$;
		(ii) $\overline{\lam\vee\mu} =\overline{\lam}\vee\overline{\mu}$; 
		(iii) $\overline{b\wedge\lam}=b\wedge\overline{\lam}$;  (iv) $\overline{\lam}[\dbt]=\cl_{\dbt}\lam[\dbt]$ and $\overline{\lam}[\dbf]=\cl_{\dbf}\lam[\dbf]$. In particular, $ \overline{1_y}(x)\geq\dbt \iff x\in\cl_\dbt\{y\}$ and $\overline{1_y}(x)\geq\dbf \iff x\in\cl_\dbf\{y\}  $ for all $x,y\in X$,  where $\cl_\dbt$ and $\cl_\dbf$ denote the closure operators of the topological spaces $(X,\tau[\dbt])$ and $(X,\tau[\dbf])$, respectively. We remind the reader that property (iii) does not hold for a general fuzzy topological space, property (iv) makes sense only for $\bbB$-valued topological spaces.

		Dually,  for all $\lam,\mu\in\mathbb{B}^X$ and all $b\in\bbB$, (i') $\lam\geq \lam^\circ$;
		(ii') $(\lam\wedge\mu)^\circ = \lam^\circ\wedge\mu^\circ$;
		(iii') $(b\ra\lam)^\circ=b\ra\lam^\circ$;
		(iv') $\lam^\circ[\dbt] =\nb_{\dbt}\lam[\dbt]$ and $\lam^\circ[\dbf]= \nb_{\dbf}\lam[\dbf]$.

		Suppose $(X,\CT)$ is a topological space. Let $$\omega(\CT) =\{\lam\in\bbB^X:\lam[\dbt]\in\CT,\lam[\dbf]\in\CT\}.$$ Then $\omega(\CT)$ is a $\bbB$-topology on $X$.  
		The assignment $(X,\CT)\mapsto(X,\omega(\CT))$ defines a full and faithful functor $$\omega\colon {\bf Top}\lra\Btop.$$ The functor $\omega$ has a right adjoint $$\iota\colon\Btop\lra{\bf Top}.$$  The functors $\omega$ and $\iota$  are special cases of the Lowen functors in fuzzy topology, see e.g. \cite{Lowen76,Kubiak}.  
		In terms of bitopological space, the functor $\omega$ sends a topological space $(X,\CT)$ to $(X,\CT,\CT)$, the functor $\iota$ sends a bitopological space $(X,\tau[\dbt], \tau[\dbf])$ to the topological space $(X,\tau[\dbt]\vee\tau[\dbf])$.  
		
		Two more functors, $$\iota_\dbt, \iota_\dbf\colon   \Btop\lra {\bf Top},$$  will be needed. For each $\bbB$-topological space $(X,\tau)$, the functor $\iota_\dbt$ sends it to $(X,\tau[\dbt])$, forgetting  $\tau[\dbf]$; the functor $\iota_\dbf$ sends it to $(X,\tau[\dbf])$, forgetting   $\tau[\dbt]$. We leave it to the reader to check that $\iota_\dbt$ has at the same time a left and a right adjoint; likewise for $\iota_\dbf$.

		\begin{defn}(\cite{Lal,BBGK}) Suppose  $P$ is a topological property and $X$ is a bitopological space. We say that \begin{enumerate}[label=\rm(\roman*)]\setlength{\itemsep}{0pt}
				\item $X$ is  join $P$ if the topological space $(X,\tau[\dbt]\vee\tau[\dbf])$  is $P$. \item $X$ is  componentwise $P$\footnote{It is called \emph{bi-$P$} in \cite{Lal,BBGK}.} if both of the topological spaces $(X,\tau[\dbt])$ and $(X,\tau[\dbf])$ are $P$. \end{enumerate} \end{defn}

		Suppose  $P$ is a topological property. It is clear that for any   topological spaces  $X$ and $Y$,   the $\bbB$-topological space $X.Y= (X\times Y,\tau)$ in Example \ref{X.Y} is join $P$  if and only if the product space $X\times Y$ is $P$.  
		
		\section{Separation axioms} \label{second section}
		
		In this section we use the $\bbB$-valued specialization  order of $\bbB$-topological spaces to postulate lower separation axioms for bitopological spaces, including $R_0,R_1,T_0,T_1$ and Hausdorff. $\bbB$-valued order is a special instance of fuzzy order, related definitions are included here for sake of self-containment.
		
		A \emph{$\bbB$-valued order} ($\bbB$-order for short) on a set $X$ is a map $r\colon X\times X\lra\bbB$ such that for all $x,y,z\in X$, \begin{enumerate}[label=$\bullet$]\setlength{\itemsep}{0pt}
			\item $r(x,x)=1$;  \item $r(y,z)\wedge r(x,y)\leq r(x,z)$. \end{enumerate}  
		
		Suppose $(X,r_X)$ and $(Y,r_Y)$ are $\bbB$-ordered sets. A map $f\colon X\lra Y$ \emph{preserves $\bbB$-order} if $r_X(x_1,x_2)\leq r_Y(f(x_1),f(x_2))$ for all $x_1,x_2\in X$. The category of $\bbB$-ordered sets and maps preserving $\bbB$-order  is denoted by  $\bbB\text{-}{\bf Ord}.$ 
		
		There is a general method to generate $\bbB$-orders. Suppose $X$ is a set and $\Lambda\subseteq\bbB^X$. Then $$\Omega(\Lambda)\colon  X\times X\lra\mathbb{B},\quad \Omega(\Lambda)(x,y)=\bigwedge\limits_{\lam\in \Lambda}(\lambda(x)\rightarrow\lambda(y)) $$ is a $\bbB$-order on $X$. Interesting instances of this method include: \begin{enumerate}[label=\rm(\roman*)]\setlength{\itemsep}{0pt}
			\item (Implication operator) If $X=\bbB$ and $\Lambda=\{\id_\bbB\}$, then  $\Omega(\Lambda)$ is the implication operator $$\ra\colon\bbB\times\bbB\lra\bbB, \quad (a,b)\mapsto \neg a\vee b.$$ 
			\item (Inclusion $\bbB$-order) For each set $X$, consider the family $\Lambda=\{[x]\colon \bbB^X\lra\bbB\}_{x\in X}$ of functions, where $[x](\lam)=\lam(x)$ for all $\lam\in\bbB^X$. Then   $\Omega(\Lambda)$   is the \emph{inclusion $\bbB$-order} $\sub_X$ on $\bbB^X$: $$\Omega(\Lambda)(\lam,\mu)=  \sub_X(\lambda,\mu)=\bigwedge\limits_{x\in X}(\lambda(x)\rightarrow\mu(x)). $$ Intuitively, the value $\sub_X(\lambda,\mu)$ measures the truth degree that $\lambda$ is contained in $\mu$, so the name.
		\end{enumerate}

		\begin{defn}Suppose $(X,\tau)$ is a $\mathbb{B}$-topological space. The $\bbB$-order $\Omega(\tau)$ on $X$ generated by $\tau$ is called the  specialization $\bbB$-order  of $(X,\tau)$. Explicitly, for all $x,y\in X$, $$\Omega(\tau)(x,y)= \bigwedge_{\lam\in\tau}(\lam(x)\ra\lam(y)).$$   \end{defn} 
		
		Some basic properties of the specialization $\bbB$-order are given in the following proposition.
		\begin{prop}\label{specialization order via component}
			Let $(X,\tau)$ be a $\bbB$-topological space. Let $\leq_\dbt$ and $\leq_\dbf$ be the specialization order of  the topological spaces $(X,\tau[\dbt])$ and $(X,\tau[\dbf])$, respectively. Then 
			\begin{enumerate}[label=\rm(\roman*)]\setlength{\itemsep}{0pt} \item For all $x,y\in X$, $$\Omega(\tau)(x,y)\geq\dbt\iff x\leq_\dbt y;\quad  \Omega(\tau)(x,y)\geq\dbf\iff x\leq_\dbf y.$$ So  the specialization $\bbB$-order of  $(X,\tau)$ determines and is determined by the specialization orders of the topological spaces $(X,\tau[\dbt])$ and $(X,\tau[\dbf])$.
				\item   $ \Omega(\tau)(x,y)=\overline{1_y}(x), $ where $\overline{1_y}$ is the closure of $1_y$ in $(X,\tau)$.  
				\item  Each subbasis $\mathcal{B}$ of $\tau$ generates $\Omega(\tau)$; that is,   $\Omega(\tau)(x,y)= \bigwedge_{\lam\in\mathcal{B}}(\lam(x)\ra\lam(y)). $   
		\end{enumerate} \end{prop}
		
		The assignment $(X,\tau)\mapsto(X,\Omega(\tau))$ gives rise to a functor $\Omega\colon \Btop\lra\bbB\text{-}{\bf Ord}$. It is known that $\Omega$ is a right adjoint functor, see e.g.  \cite{LZ2006,Zhang2018}.
		
		\begin{exmp}Suppose $X,Y$ are  topological spaces. Consider the $\bbB$-topological space $X.Y= (X\times Y,\tau)$ in Example \ref{X.Y}. Then for all $(x_1,y_1)$ and $(x_2,y_2)$ of $X\times Y$, $$ \Omega(\tau)((x_1,y_1),(x_2,y_2))=\begin{cases}1 & x_1\leq_X x_2,~y_1\leq_Y y_2, \\ \dbt & x_1\leq_X x_2,~y_1\not\leq_Y y_2, \\ \dbf & x_1\not\leq_X x_2,~y_1\leq_Y y_2, \\ 0 & x_1\not\leq_X x_2,~y_1\not\leq_Y y_2,\end{cases} $$ where $\leq_X$ and $\leq_Y$ denote the specialization order of the topological spaces $X$ and $Y$, respectively. \end{exmp}

		\subsection{\texorpdfstring{$R_0$, $T_0$ and $T_1$}{}}
		Suppose $X$ is a topological space. By the \emph{Kolmogorov equivalence relation} of $X$ we mean the relation given by $x\sim y$ if the closure of $\{x\}$ is equal to that of $\{y\}$.  It is well-known that the quotient space $X/\!\sim$ (a.k.a Kolmogorov quotient) is a $T_0$ space, it is the $T_0$-reflection of $X$.
		
		\begin{defn}(\cite{Davis}) A topological space $X$ is  $R_0$  provided that if $U$ is an open set and  $x\in U$, then the closure of $\{x\}$ is contained in $U$.\end{defn}
		
		A topological space is $R_0$ if and only if its Kolmogorov quotient is $T_1$. The following  fact should have been observed somewhere else, but we fail to locate where it is explicitly stated for the first time.

		\begin{lem}\label{R0=symmetric}   A topological space $X$ is $R_0$ if and only if its specialization order  is symmetric. \end{lem}
		
		\begin{proof} Suppose $X$ is $R_0$. If $x\not\leq y$,  the open set $X\setminus\overline{\{y\}}$ contains $x$, hence the closure of $\{x\}$, then $y\not\leq x$. This shows that the specialization order is symmetric. Conversely, suppose the specialization order of $X$ is symmetric,  $U$ is an open set and $x\in U$. For each $y\notin U$, since the closure of $\{y\}$ misses $U$, it follows that $x\not\leq y$, then $y\not\leq x$ by symmetry of the specialization order, which means $y$ is not in the closure of $\{x\}$, and consequently, $X$ is $R_0$. \end{proof}

		Lemma \ref{R0=symmetric} motivates the following postulation of $R_0$ for $\bbB$-topological spaces. 
		
		\begin{defn}Suppose $(X,\tau)$ is a $\bbB$-topological space.  We say that  \begin{enumerate}[label=\rm(\roman*)]\setlength{\itemsep}{0pt} \item  $(X,\tau)$ is  $R_0$    if   $\Omega(\tau)$ is   symmetric in the sense that $\Omega(\tau)(x,y)= \Omega(\tau)(y,x) $ for all $x,y\in X$.
				\item $(X,\tau)$ is  $T_0$    if   $\Omega(\tau)$ is   separated in the sense that $\Omega(\tau)(x,y)=1=\Omega(\tau)(y,x)\implies x=y$.  \item $(X,\tau)$ is $T_1$   if it is both $T_0$ and $R_0$. \end{enumerate} 
		\end{defn}
		
		It is readily verified that a $\bbB$-topological space $(X,\tau)$ is $R_0$ if, and only if, for each open set $\lam$ and each point $x$, it always holds that  $\lam(x)\leq\sub_X(\overline{1_x}, \lam).$  This inequality says that if $\lam$ contains $x$, then it contains the closure of $1_x$. 
		We list below some basic properties of the axioms of $R_0,T_0$ and $T_1$. The verification is straightforward, we omit it.
		
		\begin{prop}\label{T-0} Suppose $(X,\tau)$ is a $\bbB$-topological space. The following are equivalent: \begin{enumerate}[label=\rm(\arabic*)]\setlength{\itemsep}{0pt}
				\item $(X, \tau)$ is    $T_0$. 
				\item For each pair $x,y$ of distinct points of $X$, there is some $\lambda\in\tau$ such that  $\lambda(x)\neq\lambda(y)$.   \item $(X,\tau)$ is join $T_0$ in the sense of \cite{Lal,BBGK}; that is, the topological space $(X, \tau[\dbt]\vee\tau[\dbf])$ is $T_0$. \end{enumerate} 
		\end{prop}
		
		\begin{prop}\label{R0 is pairwise} Suppose $(X,\tau)$ is a $\bbB$-topological space. Then \begin{enumerate}[label=\rm(\roman*)] \setlength{\itemsep}{0pt} \item   $(X,\tau)$ is $R_0$ if and only if it is componentwise $R_0$; that is, both of the topological spaces $(X,\tau[\dbt])$ and $(X,\tau[\dbf])$ are $R_0$. \item $(X,\tau)$ is $T_1$ if and only if it is join $T_0$ and componentwise $R_0$. \end{enumerate} \end{prop} 
		
		\begin{cor}
			\begin{enumerate}[label=\rm(\roman*)] \setlength{\itemsep}{0pt} \item Subspaces and products of $T_0$   spaces are $T_0$.
				\item A topological space $X$ is $T_0$ if and only if the $\bbB$-topological  space $\omega(X)$ is $T_0$. \item Subspaces and products of $T_1$ spaces are $T_1$. \item A topological space $X$ is $T_1$ if and only if the  $\bbB$-topological  space $\omega(X)$ is $T_1$.\end{enumerate}
		\end{cor}

		Given a $\bbB$-topological space $(X,\tau)$, define an equivalence relation on $X$ by $$x\sim y\iff \forall \lam\in\tau, \lam(x)=\lam(y).$$ The equivalence relation $\sim$ is exactly the intersection of the Kolmogorov equivalence relations of $(X,\tau[\dbt])$ and $(X,\tau[\dbf])$.  The quotient $\bbB$-topological space of $(X,\tau)$ with respect to $\sim$ is $T_0$ and it is the  $T_0$-reflection of $(X,\tau)$. So, the full subcategory of $T_0$ spaces is reflective in the category of $\bbB$-topological spaces. 
		
		\begin{prop}A $\bbB$-topological space is $R_0$ if and only if its $T_0$-reflection is $T_1$.  \end{prop}
		
		By Proposition \ref{T-0}\thinspace(3),   our $T_0$ axiom coincides with join $T_0$, hence  with the $T_0$ axiom in \cite{Mu}. The following proposition together with Example \ref{T_0 of X.Y}  and Example \ref{eg2} show that our $T_1$ axiom is strictly between join $T_1$ and  componentwise $T_1$. 
		
		\begin{prop}\label{T_1} Suppose $(X,\tau)$ is a $\bbB$-topological space. 
			\begin{enumerate}[label=\rm(\roman*)] \setlength{\itemsep}{0pt} \item If $(X,\tau)$ is $T_1$, then it is join $T_1$.  
				\item If $(X,\tau)$ is componentwise $T_1$, then $(X,\tau)$ is $T_1$. \end{enumerate}
		\end{prop}
		
		\begin{proof}
			(i) Suppose that $(X,\tau)$ is a $T_1$ space. Let $x,y$ be a pair of distinct points of $X$.   Since $\Omega(\tau)$ is symmetric and separated, it follows that $\Omega(\tau)(x,y)=\Omega(\tau)(y,x)<1$. Let $\lam$ be the complement of the closure of $1_y$; that is, $\lam=\neg(\overline{1_y})$.  Then $\lambda\in\tau$, $\lambda(y)=0$  and  $\lambda(x)>0$. Hence there is a neighborhood of $x$ in $\iota(X,\tau)$ not containing $y$. Likewise, there is a neighborhood of $y$ in $\iota(X,\tau)$ not containing $x$. This shows that  $(X,\tau)$ is join $T_1$.
			
			(ii) If both $(X,\tau[\dbt])$ and $(X,\tau[\dbf])$ are $T_1$, then  the specialization $\bbB$-order  $\Omega(\tau)$ of $(X,\tau)$ is the identity relation on $X$, hence  symmetric and separated.  
		\end{proof}
		
		\begin{exmp}\label{T_0 of X.Y} Suppose $X,Y$ are  topological spaces. Consider the $\bbB$-topological space  $X.Y= (X\times Y,\tau)$  in Example \ref{X.Y}. Then \begin{enumerate}[label=\rm(\roman*)] \setlength{\itemsep}{0pt} \item $X.Y$ is $T_0$  (or equivalently, join $T_0$) if and only if  both of the topological spaces $X$ and $Y$ are $T_0$. \item $X.Y$ is $R_0$  if and only if both of the topological spaces $X$ and $Y$ are $R_0$. \item $X.Y$ is $T_1$ if and only if both of the topological spaces $X$ and $Y$ are $T_1$,    if and only if $X.Y$ is join $T_1$. \item  $X.Y$ is componentwise $T_0$ if and only if   $X$ and $Y$ are singleton spaces, if and only if $X.Y$ is componentwise $T_1$.   
			\end{enumerate} 
			
			From (iii) and (iv) it follows that  $T_1$ does not imply componentwise $T_0$. \end{exmp} 
		
		\begin{exmp}\label{eg2} Join $T_1$ does not imply $T_1$. Consider the $\bbB$-topological space $(\mathbb{R},\tau)$, where $\tau[\dbt]$ and $\tau[\dbf]$ are the left and the right topology on $\mathbb{R}$, respectively; that is,  $$\tau[\dbt]=\left\{\emptyset,\mathbb{R}\right\}\cup\left\{(-\infty,r):r\in\mathbb{R}\right\}, \quad  \tau[\dbf]=\left\{\emptyset,\mathbb{R}\right\}\cup\left\{(r,\infty):r\in\mathbb{R}\right\}.$$  Then $\tau[\dbt]\vee\tau[\dbf]$ is the usual topology on $\mathbb{R}$, hence $(\mathbb{R},\tau)$ is  join $T_1$. But, the $\bbB$-topological space $(\mathbb{R},\tau)$ is not $R_0$. To see this, take $x,y\in \mathbb{R}$ with $x<y$. Then  $\Omega(\tau)(x,y)=\overline{1_y}(x)=\dbf \not=\dbt= \overline{1_x}(y)=\Omega(\tau)(y,x).$  \end{exmp}
		
		\subsection{\texorpdfstring{$R_1$}{} and Hausdorff} \label{R1 and Hausdorff}
		
		\begin{defn}(\cite{Davis,R1space}) A topological space $X$ is $R_1$   provided that if $x$ and $y$ are points of $X$ such that $\overline{\{x\}}\not= \overline{\{y\}}$, then $x$ and $y$ have disjoint neighborhoods. \end{defn}
		
		Like $R_0$, $R_1$  can also be characterized in terms of specialization order.
		
		\begin{lem}\label{R1=closed} Let $X$ be a topological space and $\leq$ be its specialization order. The  following are equivalent: \begin{enumerate}[label=\rm(\arabic*)]\setlength{\itemsep}{0pt}
				\item $X$ is $R_1$. \item   $\leq$ is a closed set of the product space $X\times X$. \item $\leq\cap \leq^{\rm op}$ is a closed set of the product space $X\times X$. \end{enumerate} \end{lem} 
		
		\begin{proof} $(1)\Rightarrow(2)$ If $x\not\leq y$, then $\overline{\{x\}}\not= \overline{\{y\}}$, hence   $x$ and $y$ have disjoint neighborhoods, say $U$ and $V$. So $U\times V$ is a neighborhood of $(x,y)$ in the product space that misses $\leq$, then $\leq$ is a closed set of $X\times X$. 
			
			$(2)\Rightarrow(3)$ Follows from the fact that $X\times X\lra X\times X,~(x,y)\mapsto(y,x)$ is a homeomorphism. 
			
			$(3)\Rightarrow(1)$
			If $\overline{\{x\}}\not= \overline{\{y\}}$, then either $x\not\leq y$ or $y\not\leq x$, hence $(x,y)$ is not in $\leq\cap \leq^{\rm op}$. 
			Since $\leq\cap \leq^{\rm op}$ is a closed set of $X\times X$, there exist open sets $U$ and $V$ of $X$ such that $U\times V$ contains  $(x,y)$ but misses $\leq\cap \leq^{\rm op}$; in particular $U\times V$ misses the diagonal of $X\times X$. Thus, $U$ and $V$ are disjoint neighborhoods of $x$ and $y$. \end{proof}
		
		It follows from the above lemma that (i) a topological space $X$ is $R_1$ if and only if its $T_0$-reflection is Hausdorff; in particular, $X$ is Hausdorff if and only if it is both $T_0$ and $R_1$; (ii) every regular space is $R_1$; and (iii) subspaces and products of $R_1$ spaces are $R_1$.  These conclusions are already known in the literature (e.g. \cite{Davis,R1space}) and are easy to prove without resort to Lemma \ref{R1=closed}. The benefit of Lemma \ref{R1=closed} is that it  leads to the following:  
		
		\begin{defn}\label{def of Hausdorff}
			Suppose $(X,\tau)$ is a $\bbB$-topological space.  We say that \begin{enumerate}[label=\rm(\roman*)] \setlength{\itemsep}{0pt} \item  $(X,\tau)$ is   $R_1$   if   $\Omega(\tau)$ is a closed set of the product space $(X,\tau)\times(X,\tau)$. \item $(X,\tau)$ is Hausdorff  if it is both $T_0$ and $R_1$. \end{enumerate}
		\end{defn}
		
		\begin{prop} \label{Prehaus is pairwise} Suppose $(X,\tau)$ is a $\bbB$-topological space. Then $(X,\tau)$ is $R_1$ if and only if it is componentwise $R_1$; that is, both of the topological spaces $(X,\tau[\dbt])$ and $(X,\tau[\dbf])$ are $R_1$. Therefore, $(X,\tau)$ is Hausdorff if and only if it is componentwise $R_1$ and join $T_0$.
		\end{prop}
		
		\begin{proof}
			Suppose $(X,\tau)$ is $R_1$.  
			Since $\Omega(\tau)$ is a closed set of the product space $(X,\tau)\times(X,\tau)$, then $\Omega(\tau)[\dbt]= \{(x,y): \Omega(\tau)(x,y)\geq\dbt\}$ is a closed set of $\iota_\dbt((X,\tau)\times(X,\tau))= (X, \tau[\dbt])\times (X,\tau[\dbt])$. Since   $\Omega(\tau)[\dbt]$ is the specialization order  of the topological space $(X,\tau[\dbt])$, it follows that $(X,\tau[\dbt])$ is $R_1$. Likewise, $(X,\tau[\dbf])$ is $R_1$. This verifies the necessity. Sufficiency is verified in a similar way. \end{proof}
		
		Thus, a $\bbB$-topological space is Hausdorff if and only if it is join $T_0$ and componentwise $R_1$.  In particular, for each topological space $X$, the  space $\omega(X)$ is Hausdorff if and only if $X$ is Hausdorff. 
		\begin{cor}  A $\bbB$-topological space $(X,\tau)$ is   $R_1$   if and only if  $\Omega(\tau)\wedge\Omega(\tau)^{\rm op}$ is a closed set of the product space $(X,\tau)\times(X,\tau)$. \end{cor}
		
		\begin{proof}  By Proposition \ref{specialization order via component}\thinspace(i), $\Omega(\tau)[\dbt]$ and $\Omega(\tau)[\dbf]$ are  the specialization order of the topological spaces $(X,\tau[\dbt])$ and $(X,\tau[\dbf])$, respectively. Then, the conclusion follows from Proposition \ref{Prehaus is pairwise} and   Lemma \ref{R1=closed} immediately.  \end{proof}
		
		\begin{cor}  A $\bbB$-topological space is $R_1$ if and only if its $T_0$-reflection is Hausdorff. \end{cor}
		
		\begin{proof}This follows from the fact that for each topological space $X$ that is $R_1$, if $\sim$ is an equivalence relation   contained in the Kolmogorov equivalence relation,  i.e., $x\sim y$ implies $\overline{\{x\}}=\overline{\{y\}}$, 
			then the quotient space $X/\!\sim$ is $R_1$. \end{proof} 
		
		\begin{cor}  Subspaces and products of $R_1$ $\bbB$-topological spaces are $R_1$.
		\end{cor} 
		
		\begin{prop}Every $R_1$ $\bbB$-topological space is   $R_0$; every Hausdorff $\bbB$-topological space  is $T_1$. \end{prop}
		
		\begin{proof}
			If a $\bbB$-topological space  is $R_1$, then by Proposition \ref{Prehaus is pairwise} it is componentwise $R_1$, hence componentwise $R_0$, therefore  $R_0$ by Proposition \ref{R0 is pairwise}. \end{proof}
		
		\begin{prop} The full subcategory of $\Btop$ consisting of Hausdorff spaces is reflective. \end{prop}
		
		\begin{proof}
			For each $\bbB$-topological space $(X,\tau)$, let $\hat{\tau}$ be the initial $\bbB$-topology on $X$ determined by the continuous maps $f_i\colon(X,\tau)\lra (Y_i,\tau_i)$  with $(Y_i,\tau_i)$  Hausdorff. Then,   $(X,\hat{\tau})$ is an $R_1$ space and its $T_0$-reflection is the Hausdorff reflection of $(X,\tau)$. \end{proof}
		
		The following example shows, out of expectation to some extent, that Hausdorff is not preserved under refinement of $\bbB$-topologies. The reason for this is that $R_1$ is not preserved under refinement of topology. For topological spaces, $R_1$ is preserved under refinement in the realm of $T_0$ spaces. But, for $\bbB$-topological spaces, it need not be preserved even in the realm of $T_0$ spaces. Thus, Hausdorff spaces  might behave differently in the $\bbB$-valued setting.  
		
		\begin{exmp}\label{eg3} This example shows that, unlike the situation for topological spaces,  Hausdorff is not preserved under refinement of $\bbB$-topologies.
			Let $X$ be an infinite set. Let $\CT_1$ be the discrete topology,  $\CT_2$ be the indiscrete topology,  and $\CT_3$ be the co-finite  topology on $X$.   \begin{itemize}\setlength{\itemsep}{0pt} \item The specialization order of $(X,\CT_1)$ is the identity relation, so   $(X,\CT_1)$ is   $R_1$. \item The specialization order of $(X,\CT_2)$ is the trivial order (i.e., $x\leq y$ for all $x,y\in X$),  so   $(X,\CT_2)$ is $R_1$. \item Since $X$ is infinite, for any nonempty open sets $U,V$ of $(X,\CT_3)$, the set $U\times V$  meets the diagonal $\{(x,x): x\in X\}$, thus the diagonal $\{(x,x): x\in X\}$, which is specialization order of $(X,\CT_3)$, is not closed in the product of $(X,\CT_3)$ with itself. So $(X,\CT_3)$ is not $R_1$. \end{itemize} 
			The bitopological space $(X,\CT_1,\CT_2)$ is  join $T_0$ and componentwise $R_1$, so  it is   Hausdorff. The bitopological space $(X,\CT_1,\CT_3)$ is not componentwise $R_1$, it is not Hausdorff, though it is finer than $(X,\CT_1,\CT_2)$.
		\end{exmp}

		A bitopological space  $(X,\tau[\dbt], \tau[\dbf])$ is \emph{pairwise Hausdorff} \cite[Definition 2.5]{Kelly} if, for each pair of distinct points $x$ and $y$, there exist a $\tau[\dbt]$-neighborhood $U$ of $x$ and a $\tau[\dbf]$-neighborhood $V$ of $y$ such that  $U\cap V=\emptyset$.  
		In \cite{Salbany,BBGK} the axiom of pairwise Hausdorff is weakened to that  for each pair of distinct points $x$ and $y$, there exist disjoint  $U\in \tau[\dbt]$ and  $V\in\tau[\dbf]$  such that  $x\in U, y\in V$ or $x\in V,y\in U$. 
		
		It is clear that a bitopological space  $(X,\tau[\dbt], \tau[\dbf])$ is  pairwise Hausdorff in the sense of Kelly \cite{Kelly} if and only if  the diagonal $\{(x,x):x\in X\}$ is a closed set of the product space $(X,\tau[\dbt])\times(X, \tau[\dbf])$.
		
		Another generalization of the Hausdorff separation axiom to bitopological spaces is that of order-separatedness in \cite{JM2006,JM2008}. A bitopological space $(X,\tau[\dbt], \tau[\dbf])$ is \emph{order-separated}  \cite[Definition 3.7]{JM2008} if  \begin{enumerate}[label=\rm(\roman*)] \setlength{\itemsep}{0pt} \item the binary relation $\leq\thinspace \coloneqq
			\thinspace\leq_\dbt\cap \geq_\dbf$   is a partial order; and \item  $x\not\leq  y$ implies that there exist a $\tau[\dbt]$-neighborhood $U$ of $x$ and a $\tau[\dbf]$-neighborhood $V$ of $y$ such that  $U\cap V=\emptyset$.  \end{enumerate}    
		
		\begin{exmp}\label{Hasdorff of XY} Let $X,Y$ be  topological spaces. Consider the $\bbB$-topological space $X.Y= (X\times Y,\tau)$ in Example \ref{X.Y}. Then \begin{enumerate}[label=\rm(\roman*)] \setlength{\itemsep}{0pt} \item $X.Y$ is $R_1$ (Hausdorff, resp.)  if and only if so are the topological spaces $X$ and $Y$, respectively. \item $X.Y$ is pairwise Hausdorff   if and only if    $X$ and $Y$ are singleton spaces, because each nonempty open set of $(X\times Y,\tau[\dbt])$ intersects each nonempty open set of $(X\times Y,\tau[\dbf])$. \item $X.Y$ is order-separated  if and only if   $X$ and $Y$ are singleton spaces. \end{enumerate}  \end{exmp}    
		
		Example \ref{Hasdorff of XY} shows that  our axiom of Hausdorff does not imply   pairwise Hausdorff, nor order-separated.
		The space $(\mathbb{R},\tau)$ in Example \ref{eg2} is order-separated but not Hausdorff; the space $(X,\CT_1,\CT_3)$ in Example \ref{eg3}  is  pairwise Hausdorff, but not Hausdorff.  
		So, Hausdorff is neither comparable with pairwise Hausdorff nor with order-separated.

		\subsection{Regularity  and normality} \label{regular and normal}
		
		\begin{defn} Suppose $(X,\tau)$ is a $\bbB$-topological space. We say that  \begin{enumerate}[label=\rm(\roman*)] \setlength{\itemsep}{0pt} \item  $(X,\tau)$ is  regular  if for each $\lambda\in\tau$, $\lambda=\bigvee \{\mu\in\tau:\overline{\mu}\leq\lambda \}.$ \item $(X,\tau)$ is $T_3$  if it is  $T_0$ and regular. \end{enumerate}  
		\end{defn}
		The above definition is a special case of regularity and $T_3$ in \cite[Section 6]{HR80} for fuzzy topological spaces, it is motivated by the fact that in a regular topological space, every neighborhood of a point contains a closed neighborhood of the given point, see e.g. \cite[Theorem 14.3]{Willard}.
		
		\begin{prop}\label{regular is pairwise}
			A $\bbB$-topological space  $(X,\tau)$ is regular if and only if it is componentwise regular. In particular, for each topological space $X$, the $\bbB$-topological  space $\omega(X)$ is regular if and only if  $X$, as a topological space, is regular.  \end{prop}
		
		\begin{proof}
			We prove the necessity first.  Let $(X,\tau)$ be a regular $\bbB$-topological space.  Assume that $U$ is an open set of $(X,\tau[\dbt])$ and that $x\in U$. By definition of $\tau[\dbt]$, there is some $\lam\in\tau$ such that $U=\lam[\dbt]$. Since $(X,\tau)$ is regular, there is some $\mu\in\tau$ such that $\mu(x)\geq\dbt$ and $\overline{\mu}\leq\lam$. Since $\mu[\dbt]$ is an open set and $\overline{\mu}[\dbt]$ is a closed set of $(X,\tau[\dbt])$, it follows that $\mu[\dbt]$ is an open neighborhood of $x$ whose closure is contained in $U$, hence $(X,\tau[\dbt])$  is regular.  Likewise,  $(X,\tau[\dbf])$ is regular.
			
			For sufficiency, suppose that $(X,\tau[\dbt])$ and $(X,\tau[\dbf])$ are both regular. Let $\lambda\in\tau$.  Since   $\lambda[\dbt]$ is  open in $(X, \tau[\dbt])$, there is a family $\mathcal{U}$ of open sets of $(X, \tau[\dbt])$ such that $\mathcal{U}$ covers $\lam[\dbt]$ and the closure of each of its element in $(X, \tau[\dbt])$ is contained in $\lam[\dbt]$. Similarly, there is a family $\mathcal{V}$ of open sets of $(X, \tau[\dbf])$ such that $\mathcal{V}$ covers $\lam[\dbf]$ and the closure of each of its element in $(X, \tau[\dbf])$ is contained in $\lam[\dbf]$.   Let $$\Lambda=\{\dbt_U: U\in\mathcal{U}\}\cup\{\dbf_V:V\in\mathcal{V}\}.$$ Then $\Lambda$ is a family of open sets of $(X,\tau)$ such that $\lam$ is the join of $\Lambda$ and that the closure of each element of $\Lambda$   is contained in $\lambda$. This shows that $(X,\tau)$ is regular.
		\end{proof}

		\begin{cor}\label{regular} \begin{enumerate}[label=\rm(\roman*)] \setlength{\itemsep}{0pt} \item A $\bbB$-topological space   is $T_3$ if and only if it is  join $T_0$ and  componentwise regular.  \item  Every   $\bbB$-topological space that is regular is $R_1$; every $\bbB$-topological space that is $T_3$ is Hausdorff.  \end{enumerate}  
		\end{cor}
		
		\begin{defn}Suppose  $(X,\tau)$ is a $\bbB$-topological space. We say that \begin{enumerate}[label=\rm(\roman*)] \setlength{\itemsep}{0pt} \item (\cite{Hutton75}) $(X,\tau)$ is  normal if for each open set $\lambda $ and each closed set $\mu$ of $(X,\tau)$ with $\mu\leq\lam$, there is an open set $\nu$ such that $\mu\leq\nu\leq\overline{\nu}\leq\lam$. \item $(X,\tau)$ is $T_4$  if it is  $T_1$ and normal. \end{enumerate}  
		\end{defn}
		
		\begin{prop}\label{normal=binormal} \begin{enumerate}[label=\rm(\roman*)] \setlength{\itemsep}{0pt} \item A $\bbB$-topological space $(X,\tau)$ is normal if and only if it is componentwise normal.  \item A $\bbB$-topological space $(X,\tau)$ is $T_4$ if and only if it is  join $T_0$, componentwise $R_0$ and componentwise normal.   \end{enumerate} 
		\end{prop}
		
		\begin{proof} Similar to  Proposition \ref{regular is pairwise}. \end{proof} 
		
		As an immediate corollary we obtain that for each topological space $X$, the space $\omega(X)$ is normal space if and only if $X$, as a topological space, is   normal.

		There is a Urysohn lemma for normal $\bbB$-topological spaces, which is essentially a special case of that for normal fuzzy topological spaces in Hutton \cite{Hutton75}. We formulate it in terms of bitopological spaces. We need an analog of the unit interval in the bitopological setting:  $$[0,1](\bbB)=([0,1]\times[0,1],\tau[\dbt], \tau[\dbf]),$$ where, $\tau[\dbt]=\{U\times[0,1]: U\in\CO\}$, $\tau[\dbf]=\{[0,1]\times V: V\in\CO\}$,  $\CO$ is the usual topology of $[0,1]$.  In the notation of Example \ref{X.Y},   $[0,1](\bbB)$ is  the space $[0,1].[0,1]$. 
		Now we state the Urysohn lemma, leaving the proof to the reader.
		\begin{lem}[Urysohn lemma] Suppose $(X,\tau[\dbt],\tau[\dbf])$ is a normal bitopological space. If \begin{enumerate}[label=\rm(\roman*)] \setlength{\itemsep}{0pt} \item $F_\dbt$ is a closed set of $(X,\tau[\dbt])$, $F_\dbf$ is a closed set of $(X,\tau[\dbf])$, \item $U_\dbt$ is an open set of $(X,\tau[\dbt])$, $U_\dbf$ is an open set of $(X,\tau[\dbf])$, and \item $F_\dbt\subseteq U_\dbt$, $F_\dbf\subseteq U_\dbf$,\end{enumerate} then there exists a continuous map $f\colon X\lra [0,1](\bbB)$ such that  $$\forall x\in F_\dbt,~ f(x)\in\{0\}\times[0,1];\quad \forall x\in F_\dbf,~ f(x)\in  [0,1]\times\{0\}$$ and $$\forall x\notin U_\dbt,~ f(x)\in\{1\}\times[0,1];\quad \forall x\notin U_\dbf,~ f(x)\in  [0,1]\times\{1\}.$$\end{lem}
		
		\begin{prop} If a $\bbB$-topological space  is $T_4$, then it is $T_3$.
		\end{prop}
		
		\begin{proof}
			Suppose $(X,\tau)$ is  $T_4$. To see that it is $T_3$, by Proposition \ref{regular is pairwise} it suffices to show that both of the topological spaces $(X,\tau[\dbt])$ and $(X,\tau[\dbf])$ are regular. Assume that $F$ is a closed set of $(X,\tau[\dbt])$ and $x\notin F$. Since $(X,\tau[\dbt])$ is $R_0$ by Proposition \ref{R0 is pairwise}, it follows that in $(X,\tau[\dbt])$, the closure of $\{x\}$ is contained in $X\setminus F$, hence disjoint with $F$.  Since $(X,\tau[\dbt])$ is a normal space by Proposition \ref{normal=binormal}, there are disjoint open sets of $(X,\tau[\dbt])$ that contain $x$ and $F$, respectively. This shows that $(X,\tau[\dbt])$ is regular. Likewise, $(X,\tau[\dbf])$ is regular. 
		\end{proof}    
		
		\begin{defn} (\cite{Kelly}) Suppose $(X,\tau[\dbt], \tau[\dbf])$ is a bitopological space.  
			\begin{enumerate}[label=\rm(\roman*)] \setlength{\itemsep}{0pt}  
				\item If for each $\tau[\dbt]$-closed set $F$ and each $x\notin F$, there exist a $\tau[\dbt]$-open set $U$ and $\tau[\dbf]$-open set $V$ such that $x\in U$, $F\subseteq V$, and $U\cap V=\emptyset$, then we say that $\tau[\dbt]$ is regular with respect to $\tau[\dbf]$. If $\tau[\dbt]$ is regular with respect to $\tau[\dbf]$, and $\tau[\dbf]$ is regular with respect to $\tau[\dbt]$, then we say that $(X,\tau[\dbt],\tau[\dbf])$ is pairwise regular.
				
				\item If for each $\tau[\dbt]$-closed set $F$ and $\tau[\dbf]$-closed set $G$ with $F\cap G=\emptyset$,   there exist a $\tau[\dbt]$-open set $U$ and a $\tau[\dbf]$-open set $V$ such that $F\subset V$, $G\subset U$, and $U\cap V=\emptyset$, then we say that $(X,\tau[\dbt], \tau[\dbf])$ is  pairwise normal. \end{enumerate} \end{defn}
		
		\begin{exmp}Suppose $X,Y$ are  topological spaces. Consider the $\bbB$-topological space $X.Y= (X\times Y,\tau)$ in Example \ref{X.Y}. Then \begin{enumerate}[label=\rm(\roman*)] \setlength{\itemsep}{0pt} \item $X.Y$ is regular (resp. normal) if and only if so are the topological spaces $X$ and $Y$, respectively. \item $X.Y$ is pairwise regular   if and only if  both of the topological spaces $X$ and $Y$ are indiscrete, because each nonempty open set of $(X\times Y,\tau[\dbt])$ intersects each nonempty open set of $(X\times Y,\tau[\dbf])$. \item $X.Y$ is pairwise normal trivially, because each nonempty closed set of $(X\times Y,\tau[\dbt])$ intersects each nonempty closed set of $(X\times Y,\tau[\dbf])$. \end{enumerate} 
		\end{exmp}

		Regularity and  normality of bitopological spaces postulated here are different  from the pairwise regularity and the pairwise normality of Kelly \cite{Kelly}. The space $(\mathbb{R},\tau)$ in Example \ref{eg2} is pairwise regular space but not regular. For a $\bbB$-topological spaces that is pairwise normal but not normal, consider the   space $(X,\tau)$, where    $$X=\{x,y,z\}, \quad\tau[\dbt] =\{\emptyset,\{x\},X\},\quad \tau[\dbf]=\{\emptyset,\{x\},\{x,y\},\{x,z\},X\}.$$  Then   $(X,\tau)$ is  pairwise normal but not normal.  Finally, consider the $\bbB$-topological space $(X,\tau)$, where    $$X=\{x,y,z\}, \quad \tau[\dbt]= \{\emptyset,\{x,y\},\{z\},X\} ,\quad \tau[\dbf]=\{\emptyset,\{x\} ,\{y,z\} ,X\}.$$  A calculation gives  \begin{center} \begin{tabular}{   c| c| c |c}
				$\Omega(\tau)$   & \thinspace$x$\thinspace & \thinspace$y$\thinspace & \thinspace$z$\thinspace \\   \hline  
				$x$   & $1$ &$\dbt$ &$0$ \\ \hline
				$y$   & $\dbt$ & $1$ & $\dbf$ \\ \hline
				$z$   & $0$ & $\dbf$& $1$\\
		\end{tabular}\end{center} which shows that $(X,\tau)$ is $T_1$. Since $(X,\tau)$ is   componentwise normal,   it is   $T_4$, hence $T_3$ and Hausdorff.  But,  $(X,\tau)$  is neither  pairwise regular, nor pairwise normal, nor pairwise Hausdorff.

		Interrelationship between the separation axioms in this section are summarized in the diagram: 
		
		
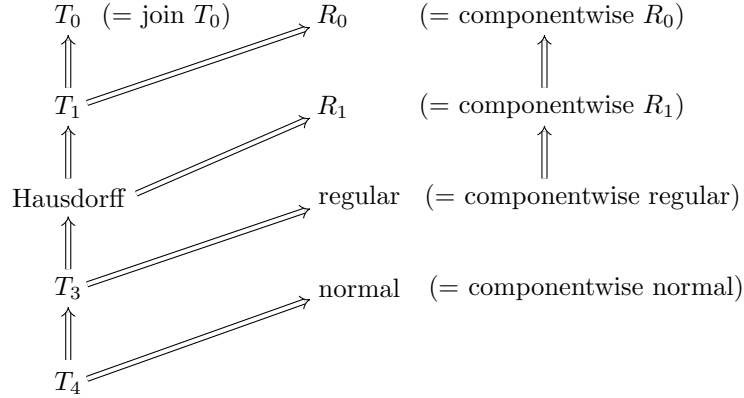
\begin{center}\begin{tikzpicture}  
				\draw[-{Implies},double distance = 1.5pt] (0,0) -- (0,0.7); 
				\draw[-{Implies},double distance = 1.5pt] (0,1.2) -- (0,1.9);
				\draw[-{Implies},double distance = 1.5pt] (0,2.4) -- (0,3.1);
				\draw[-{Implies},double distance = 1.5pt] (0,3.6) -- (0,4.3);
				
				\node at (0,-0.3) {$T_4$}; \node at (0,0.95) {$T_3$};\node at (0,2.15) {Hausdorff}; \node at (0,3.35) {$T_1$}; \node at (0,4.55) {$T_0$}; \node at (1.3,4.55) {(= join $T_0$)}; \node at (3.5,4.55) {$R_0$}; \node at (3.5,3.35) {$R_1$};\node at (3.85,2.15) {regular}; \node at (3.85,0.95) {normal}; 
				
				\draw[-{Implies},double distance = 1.5pt] (0.25,-0.25) -- (3.2,0.8); 
				\draw[-{Implies},double distance = 1.5pt] (0.25,1.0) -- (3.2,2); 
				\draw[-{Implies},double distance = 1.5pt] (0.9,2.2) -- (3.2,3.2); 
				\draw[-{Implies},double distance = 1.5pt] (0.25,3.4) -- (3.2,4.4);
				\node at (6.4,4.55) {(= componentwise $R_0$)}; 
				\node at (6.4,3.35) {(= componentwise $R_1$)};  \node at (6.85,0.95) {(= componentwise normal)};  \node at (6.8,2.15) {(= componentwise regular)}; 
				\draw[-{Implies},double distance = 1.5pt] (6.3,2.4) -- (6.3,3.1); 
				\draw[-{Implies},double distance = 1.5pt] (6.3,3.6) -- (6.3,4.3);
		\end{tikzpicture} \end{center}
		
		\section{Sobriety} \label{third section}
		
		Sobriety is an important notion in topology, as  manifested by its role in the connection between topological spaces and frames, the algebraic dual of topological spaces  \cite{Johnstone,PP2012}. In 1983, Banaschewski, Br\"{u}mmer, and Hardie \cite{BBH} introduced biframes as algebraic duals of bitopological spaces. In 2006, Jung and Moshier  \cite{JM2006,JM2008} developed an algebraic realization of bitopological spaces, resulting in the theory of of d-frames and d-sober bitopological spaces. Since bitopological spaces are a special kind of fuzzy topological spaces, $\bbB$-valued topological spaces to be precise, the theory of sobriety developed for fuzzy topological spaces, see e.g. \cite{Zhang2018,ZL95,Zhang2022}, can be transformed into a theory of sobriety for bitopological spaces. This section investigates the relationship between d-sobriety in \cite{JM2006,JM2008} and $\bbB$-sobriety in \cite{Zhang2018,ZL95}, and establishes a Hofmann-Mislove theorem for bitopological spaces.
		
		\subsection{d-sobriety}
		
		A \emph{frame} is a complete lattice $L$ such that for each element $a$ and each subset $\{b_i\}_{i\in I}$ of $L$, it holds that $$a\wedge\bv_{i\in I}b_i=\bv_{i\in I}a\wedge b_i.$$ A \emph{frame homomorphism} between frames is a map that preserves finite meets and arbitrary joins. With frames as objects and frame homomorphisms as morphisms we have a category {\bf Frm}, called the category of frames.  For the theory of  frames we refer to the monographs \cite{Johnstone,PP2012}.

		Suppose $(X,\tau)$ is a $\bbB$-topological space. The open-set lattice $\tau$ is not only a frame, it is an object in the slice category $\BFrm$. Actually, assigning to each element $b$ of $\bbB$ the constant open set $b_X$ gives a frame map $i_X\colon\bbB\lra\tau$, so $(\tau,i_X)$ is an object of $\BFrm$. The images of $\dbt,\dbf$ under $i_X$, which are also denoted by $\dbt,\dbf$ respectively, are a complemented pair of $\tau$. Then, ``a second distributive lattice that is `at 90 degrees' to the frame order
		also exists'' on $\tau$ \cite[page 31]{JM2006}. Precisely, by defining $$x\sqcap y=(x\wedge\dbf)\vee(y\wedge\dbf)\vee(x\wedge y)$$ $$x\sqcup y=(x\wedge\dbt)\vee(y\wedge\dbt)\vee(x\wedge y)$$ we obtain a distributive lattice $(\tau,\sqcap,\sqcup)$ with $\dbt$ being the top element and $\dbf$ being the bottom element, see \cite[page 751]{BK47} or \cite[Section 3.2]{JM2006}. The order $\sqsubseteq$ of the lattice $(\tau,\sqcap,\sqcup)$ is called the \emph{logic order} of $(\tau,i_X)$ in \cite{Jakl,JM2006, JM2008}.\footnote{We have flipped notations from   main references for d-frames \cite{Jakl,JM2006,JM2008}, the symbols $\sqsubseteq,\sqcap$ and $\sqcup$, instead of $\leq,\wedge$ and $\vee$, are used for the logic order here.}
		
		\begin{defn}(\cite[Definition 2.3.2]{Jakl})
			A d-frame is a structure $\bbL=(L;\dbt,\dbf;\mathbf{con},\mathbf{tot})$, where $L$ is a frame, $\dbt$ and $\dbf$ are a pair of elements of $L$ that are complement of each other, $\mathbf{con}$ and $\mathbf{tot}$ are subset of $L$, called the consistency relation and the totality relation respectively,  subject to the following conditions: 
			\begin{itemize}
				\item $\mathbf{con}$ is a Scott closed set and $\mathbf{tot}$ is an upper set of $(L,\leq)$; that is, 
				
				($\mathbf{con}$--$\downarrow$) $\alpha\leq\beta$ \& $\beta\in\mathbf{con}\implies \alpha\in\mathbf{con}$
				
				($\mathbf{con}$--$\bigvee^\uparrow$) $A\subset\mathbf{con}$ is directed w.r.t. $\leq\implies \bigvee  A\in\mathbf{con}$
				
				($\mathbf{tot}$--$\uparrow$) $\alpha\leq\beta$ \& $\alpha\in\mathbf{tot}\implies \beta\in\mathbf{tot}$;
				
				\item   $\mathbf{con}$ and $\mathbf{tot}$ are sublattices of $(L;\sqcap,\sqcup)$; that is, 
				
				($\mathbf{con}$--$\sqcap,\sqcup$) $\alpha,\beta\in\mathbf{con}\implies \alpha\sqcap\beta,\alpha\sqcup\beta\in\mathbf{con}$ 
				
				($\mathbf{tot}$--$\sqcap,\sqcup$) $\alpha,\beta\in\mathbf{tot}\implies \alpha\sqcap\beta,\alpha\sqcup\beta\in\mathbf{tot}$;
				
				\item ($\mathbf{con}$--$\dbt,\dbf$) $\dbt,\dbf\in\mathbf{con}$ 
				
				($\mathbf{tot}$--$\dbt,\dbf$) $\dbt,\dbf\in\mathbf{tot}$;
				
				\item ($\mathbf{con}$--$\mathbf{tot}$) $\alpha\in\mathbf{con}$, $\beta\in\mathbf{tot}$,  $(\alpha\wedge\dbt=\beta\wedge\dbt\text{ or }\alpha\wedge\dbf= \beta\wedge\dbf)\implies \alpha\leq\beta$.
			\end{itemize} 
		\end{defn}
		
		D-frames in the above definition are called   \emph{reasonable d-frames} by their inventors Jung and Moshier \cite{JM2006,JM2008}. We follow Jakl \cite{Jakl}  and call them simply \emph{d-frames}.
		
		\begin{exmp} (\cite{JM2006,JM2008}) The Boolean algebra $\bbB=\{1,0,\dbt,\dbf\}$ can be made into a d-frame in a unique way: $\mathbf{tot}=\{\dbt,\dbf,1\}$, $\mathbf{con}=\{\dbt,\dbf,0\}$. We always view $\bbB$ as a d-frame in this subsection. \end{exmp}
		
		Let $\bbL=(L;\dbt_L,\dbf_L; \mathbf{con}_L,\mathbf{tot}_L)$ and $\mathbb{M}= (M;\dbt_M,\dbf_M;\mathbf{con}_M,\mathbf{tot}_M)$ be d-frames.  A d-frame homomorphism  $f\colon \bbL\lra \mathbb{M}$   is a frame homomorphism $f\colon L\lra M$ that preserves $\dbt,\dbf,\mathbf{con}$ and $\mathbf{tot}$, i.e.,  $$f(\dbt_L)=\dbt_M,~f(\dbf_L)=\dbf_M,~f(\mathbf{con}_L)\subseteq\mathbf{con}_M,~ f(\mathbf{tot}_L)\subseteq \mathbf{tot}_M. $$ The category of d-frames and d-frame homomorphisms is denoted by  $\dFrm.$ 
		
		Suppose $(L;\dbt,\dbf; \mathbf{con},\mathbf{tot})$ is a d-frame. It is clear that   $[0,\dbt]$ and $[0,\dbf]$ are frames. Furthermore, $L$  is   isomorphic to the product $[0,\dbt]\times[0,\dbf]$ via the correspondence $$\alpha\mapsto (\alpha\wedge\dbt,\alpha\wedge\dbf).$$  Hence, a d-frame homomorphism $f\colon \bbL\lra \mathbb{M}$ is determined by its restriction on $[0,\dbt]$ and $[0,\dbf]$, which gives us two  frame homomorphisms.   
		
		For each bitopological space $(X,\tau[\dbt],\tau[\dbf])$, it is readily verified that $$(\tau[\dbt]\times \tau[\dbf];(X,\emptyset),(\emptyset,X);\mathbf{con},\mathbf{tot})$$ is a d-frame, where \begin{align*}\mathbf{con}&=\{(U,V)\in\tau[\dbt]\times\tau[\dbf]:U\cap V=\emptyset\},\\ \mathbf{tot}&=\{(U,V)\in\tau[\dbt]\times\tau[\dbf]:U\cup V=X\}.\end{align*} For each continuous map $f\colon(X,\tau_X[\dbt],\tau_Y[\dbf])\lra (Y,\tau_Y[\dbt],\tau_Y[\dbf])$  between bitopological spaces,  $$ {f}^{-1}\colon \tau_Y[\dbt]\times \tau_Y[\dbf]\lra \tau_X[\dbt]\times \tau_X[\dbf], \quad (U,V)\mapsto (f^{-1}(U),f^{-1}(V))$$ is a d-frame homomorphism. These are the constituents of the  functor 
		$$\dO\colon\bf{BiTop}\lra \dFrm^{\rm op}.$$
		
		As in the situation for topological spaces and frames, the functor $\dO$ has a right adjoint.
		
		\begin{defn}(\cite[Definition 3.4]{JM2008})
			A  d-point of a d-frame $\mathbb{L}=(L;\dbt,\dbf; \mathbf{con},\mathbf{tot})$  is a d-frame homomorphism $p\colon \mathbb{L}\lra\bbB$.  \end{defn} 
		
		Since $L$ is isomorphic to the product $[0,\dbt]\times[0,\dbf]$, a d-point of a d-frame $\bbL$ is essentially a pair of frame homomorphisms $$p_\dbt\colon[0,\dbt]\lra\{0,\dbt\}\quad\text{and}\quad p_\dbf\colon[0,\dbf]\lra\{0,\dbf\}$$ subject to the following conditions:  \begin{itemize} \item[]  $(\text{dp}_{\mathbf{con}})$ $\alpha \in \mathbf{con} \implies p_\dbt(\alpha\wedge\dbt)=0$ or $p_\dbf(\alpha\wedge\dbf)=0$; \item[] $(\text{dp}_{\mathbf{tot}})$ $\alpha \in \mathbf{tot} \implies p_\dbt(\alpha\wedge\dbt)=\dbt$ or $p_\dbf(\alpha\wedge\dbf)=\dbf$. \end{itemize}
		
		The set ${\rm dpt}(\bbL)$ of d-points of a d-frame $\bbL$ becomes a bitopological space by considering  as the topology $\tau[\dbt]$ the collection of $$  \Phi_\dbt(a) =\{p :p(a)= \dbt\},\quad a\in [0,\dbt],$$   and as the   topology $\tau[\dbf]$ the collection of $$\Phi_\dbf(b) =\{p:p(b)= \dbf\},\quad b\in [0,\dbf].$$    
		The construction for objects is extended to a  functor 
		$${\rm dpt}\colon \dFrm^{\rm op}\lra \bf{BiTop}$$ in the usual way \cite{JM2006,JM2008}.
		
		\begin{thm}{\rm(\cite[Theorem 3.5]{JM2008})}
			${\rm dpt}\colon \dFrm^{\rm op}\lra \bf{BiTop}$ is right adjoint to  $\dO\colon\bf{BiTop}\lra \dFrm^{\rm op}$.  
		\end{thm}
		
		Suppose  $(X,\tau[\dbt],\tau[\dbf])$ is a bitopological space and   $x\in X$. Then $$[x]\colon\tau[\dbt]\times\tau[\dbf]\lra\bbB$$  
		is a d-point of the d-frame $\dO(X,\tau[\dbt],\tau[\dbf])$, where $$ [x](U,V) \geq\dbt \iff x\in U, \quad   [x](U,V) \geq\dbf \iff x\in V.$$  
		
		\begin{defn}(\cite{JM2006,JM2008}) \label{sober bitop}
			A bitopological space $(X,\tau[\dbt],\tau[\dbf])$ is d-sober if for each d-point $p$ of  the d-frame $$\dO(X,\tau[\dbt],\tau[\dbf]),$$  there is a unique point $x$ of $X$ such that $p=[x]$.  
		\end{defn}
		
		A sober bitopological space is  necessarily $T_0$. As in the situation for topological spaces, a bitopological $X$ is sober if and only if the component $X\lra{\rm dpt}\circ\dO(X)$  of the unit of the adjunction $\dO\dashv{\rm dpt}$ at $X$ is a bijection, hence a homeomorphism. The reader is referred to \cite{Jakl,JM2006,JM2008} for more information on d-sober spaces. In particular,   every d-sober bitopological space is join sober \cite[Lemma 4.6]{JM2006}, hence sober in the sense of \cite{BBH}; a componentwise sober bitopological space need not be d-sober \cite[Counterexample 4.5]{JM2006}. 
		
		It is well-known that a Hausdorff topological space is a sober space. 
		Theorem 3.9 in \cite{JM2008} shows that every order-separated bitopological space is d-sober. The following theorem shows that every Hausdorff  bitopological space in the sense of Definition \ref{def of Hausdorff} is d-sober. We remind the reader that order-separated  and Hausdorff do not imply each other, see   Subsection \ref{R1 and Hausdorff}.
		
		\begin{prop}\label{Hd-s}
			Every Hausdorff  bitopological space is d-sober.
		\end{prop}
		
		\begin{proof}
			Suppose   $(X,\tau[\dbt],\tau[\dbf])$ is a Hausdorff bitopological space. 	Let $p$ be a d-point of the d-frame $\dO(X,\tau[\dbt],\tau[\dbf])$. Then,  
			the complement $K_\dbt$ of $\bigcup\{U\in\tau[\dbt]:p(U) =0\}$ is an irreducible closed set of $(X,\tau[\dbt])$, the complement $K_\dbf$ of $\bigcup\{V\in\tau[\dbf]:p(V) =0\}$ is an irreducible closed set of $(X,\tau[\dbf])$. Since $(X,\tau[\dbt])$ is $R_1$, it is readily verified that $K_\dbt$ is the closure   of each of its element in $(X,\tau[\dbt])$; that is, $K_\dbt=\cl_{\tau[\dbt]}\{x\}$ for all $x\in K_\dbt$.  Likewise, $K_\dbf$ is the closure   of each of its element in $(X,\tau[\dbf])$. 
			Since $p(X\setminus K_\dbt)=0$ and $p(X\setminus K_\dbf)=0$, it follows  
			that $K_\dbt\cap K_\dbf\not=\emptyset$. Any point $x$ of $K_\dbt\cap K_\dbf$ satisfies  $[x]=p$. Uniqueness of $x$ follows from that $(X,\tau[\dbt], \tau[\dbf])$ is $T_0$.
		\end{proof} 
		 
	\subsection{\texorpdfstring{$\mathbb{B}$}{}-sobriety}
	
	Let $G\colon \dFrm\lra\mathbb{B}\downarrow\mathbf{Frm}$ denote the forgetful functor, which forgets the consistency and the totality relation of d-frames. Explicitly, $G$ maps a d-frame  $ (L;\dbt, \dbf; \mathbf{con},\mathbf{tot})$ to $(L,i_L)$, where  $i_L\colon \bbB\lra L$ is the frame homomorphism given by $i_L(\dbt)=\dbt$ and $i_L(\dbf)=\dbf$.  
	
	For an object $(L,i_L)$ of the slice category $\BFrm$, we also write $\dbt$ and $\dbf$ for the elements $i_L(\dbt)$ and $i_L(\dbf)$ of $L$, respectively. This will cause no confusion. Then $$F(L,i_L)\coloneqq(L;  \dbt,\dbf;  \mathbf{con},\mathbf{tot})$$ is a d-frame, where $\mathbf{con}= \thinspace\downarrow\!\dbt\thinspace\cup\downarrow\!\dbf$ and $\mathbf{tot}=\thinspace\uparrow\!\dbt\thinspace \cup\uparrow\!\dbf$. In this way we obtain a functor $$F\colon\BFrm\lra\dFrm.$$ 
	
	\begin{prop} $F\colon\BFrm\lra\dFrm$ is left adjoint to  $G\colon\dFrm\lra\BFrm$.
	\end{prop}
	
	\begin{proof}Routine verification. \end{proof}
	
	The composite of the adjunctions $\dO\dashv{\rm dpt}$ and $G^{\rm op}\dashv F^{\rm op}$ gives a dual adjunction between the category of bitopological spaces and the slice category $\BFrm$. This adjunction is essentially a special case of the adjunction in \cite[Lemma 1.3]{ZL95} between a category of frame-valued topological spaces and a slice category of frames. In the following we formulate this adjunction in terms of $\bbB$-topological spaces instead of bitopological spaces. 
	
	The assignment $(X,\tau)\mapsto (\tau,i_X)$ defines a functor $$\BO\colon\Btop\lra(\BFrm)^{\rm op},$$ where $i_X\colon\bbB\lra\tau$ is the frame homomorphism  that sends each $b\in\bbB$ to the constant open set $b_X$.
	
	Suppose $(L,i_L)$ is an object in the slice category $\BFrm$. By a \emph{$\bbB$-point} of $(L,i_L)$ we mean a morphism $p\colon(L,i_L)\lra(\bbB,\id_\bbB)$ in $\BFrm$; in other words,  $p\colon L\lra\bbB$ is a frame homomorphism such that $p\circ i_L=\id_\bbB$. On the set of all $\bbB$-points of $(L,i_L)$ we consider the $\bbB$-topology $\{\phi(a):a\in L\}$,  where $\phi(a)(p)= p(a)$ for all $a\in L$ and all $\bbB$-point $p$. In this way we obtain a functor $$\Bpt\colon(\BFrm)^{\rm op}\lra\Btop.$$ 
	
	\begin{prop} {\rm(\cite{ZL95})}  $\Bpt\colon(\BFrm)^{\rm op}\lra\Btop$ is right adjoint to  $\BO\colon\Btop\lra(\BFrm)^{\rm op}$.
	\end{prop}
	
	\begin{proof} This follows from the fact that $\Bpt$ is the composite  ${\rm dpt}\circ F^{\rm op}$ of right adjoints, and $\BO$  is the composite  $G^{\rm op}\circ \dO$ of left adjoints. A direct verification is also easy, as in \cite{ZL95}. \end{proof} 
	
	\begin{defn}(\cite{ZL95}) 
		A $\bbB$-topological space $(X,\tau)$ is \emph{$\bbB$-sober} if for each $\bbB$-point $p$ of  $(\tau,i_X)$, there is a unique point $x$ of $X$ such that $p(\lam)=\lam(x)$ for all $\lam\in\tau$. \end{defn}
	
	While d-sobriety arises  from the adjunction $\dO\dashv{\rm dpt}$, with d-frames being viewed as algebraic duals of bitopological spaces; $\bbB$-sobriety arises from the adjunction $\BO\dashv\Bpt$, with objects in the slice category $\BFrm$ being viewed as algebraic duals of bitopological spaces. Explicitly, a $\bbB$-topological space $(X,\tau)$  is  $\bbB$-sober  if the component $(X,\tau)\lra\Bpt\circ\BO(X,\tau)$ of the unit of the adjunction $\BO\dashv\Bpt$ at $(X,\tau)$ is a bijection, hence a homeomorphism. 
	
	\begin{exmp}The space $(\bbB,\tau_\bbS)$ in Example \ref{Sierpinski} is $\bbB$-sober. Furthermore, the arguments in \cite{Nel,Noor} can be applied to show that the category of $\bbB$-sober spaces is the epireflective hull of the space $(\bbB,\tau_\bbS)$ in the category $\Btop_0$ of  $\bbB$-topological spaces satisfying the $T_0$ axiom.  
	\end{exmp}  
	
	$\bbB$-sobriety can be characterized in terms of irreducible closed sets. 
	
	\begin{defn}(\cite{Zhang2018}) Let $(X,\tau)$ be a $\bbB$-topological space. A closed set $\gamma$ of $(X,\tau)$ is irreducible if \begin{enumerate}[label=\rm(\roman*)] \setlength{\itemsep}{0pt} \item $\sub_X(\gamma,b_X)=b$ for all $b\in\bbB$;   \item   $\sub_X(\gamma,\mu_1\vee\mu_2)= \sub_X(\gamma,\mu_1)\vee\sub_X(\gamma,\mu_2$) for all closed sets $\mu_1, \mu_2$ of $(X,\tau)$.  \end{enumerate} \end{defn}
	Condition (i) is equivalent to requiring $\gamma$ be \emph{inhabited} in the sense that $\bigvee_{x\in X}\gamma(x)=1$.
	
	Suppose $(X,\tau[\dbt],\tau[\dbf])$ is a bitopological space. We say that $(K_\dbt,K_\dbf)$ is  a \emph{pair of irreducible closed sets}  of $(X,\tau[\dbt], \tau[\dbf])$ if, $K_\dbt$ is an irreducible closed set of $(X,\tau[\dbt])$ and  $K_\dbf$ is an irreducible closed set of $(X,\tau[\dbf])$.  The following lemma  says that irreducible closed sets of a $\bbB$-topological space $(X,\tau)$ correspond to pairs of irreducible closed sets of the bitopologial space $(X,\tau[\dbt], \tau[\dbf])$.
	
	\begin{lem}\label{characterizing irreducible sets} Suppose $\gamma$ is a closed set of a $\bbB$-topological space  $(X,\tau)$. Then,  $\gamma$ is irreducible if  and only if $(\gamma[\dbt], \gamma[\dbf])$ is a pair of irreducible closed sets of the bitopological space $(X,\tau[\dbt],\tau[\dbf])$. \end{lem}
	
	\begin{proof} For necessity, suppose  $\gamma$ is an irreducible  closed set of $(X,\tau)$. We wish to show that $\gamma[\dbt]$ is an irreducible closed set of $(X,\tau[\dbt])$  and $\gamma[\dbf]$ is an irreducible closed set of $(X,\tau[\dbf])$. That $\gamma[\dbt]\not=\emptyset$ is obvious. Suppose on the contrary that $\gamma[\dbt]$ is not an irreducible closed subset of $(X,\tau[\dbt])$. Then there are closed sets $A,B$ of $(X,\tau[\dbt])$ such that $\gamma[\dbt]\nsubseteq A$, $\gamma[\dbt]\nsubseteq B$, and $\gamma[\dbt]\subseteq A\cup B$. Let $$\mu_1=\dbt_A\vee(\dbf\wedge\gamma) \quad \text{and}\quad \mu_2= \dbt_B\vee(\dbf\wedge\gamma).$$ Then  $\sub_X(\gamma,\mu_1\vee\mu_2)=1$, but   $\sub_X(\gamma,\mu_1)\vee \sub_X(\gamma,\mu_2)\leq\dbf$, a contradiction. Likewise, $\gamma[\dbf]$ is an irreducible closed set of $(X,\tau[\dbf])$.
		
		For sufficiency,  suppose $\gamma$ is a closed set of $(X,\tau)$ such that $\gamma[\dbt]$ is an irreducible closed set of $(X,\tau[\dbt])$ and $\gamma[\dbf]$ is an irreducible closed set of $(X,\tau[\dbf])$. Since both $\gamma[\dbt]$ and $\gamma[\dbf]$ are not empty, then  $\bigvee_{x\in X}\gamma(x)=1$. 
		It remains to check that for all closed sets $\mu_1,\mu_2$ of $(X,\tau)$, $$\sub_X(\gamma,\mu_1\vee\mu_2) = \sub_X(\gamma,\mu_1)\vee\sub_X(\gamma,\mu_2).$$  We proceed with four cases. 
		
		Case 1.  $\sub_X(\gamma,\mu_1)=1$ or $\sub_X(\gamma,\mu_2)=1$. 
		In this case there is nothing to prove.
		
		Case 2.   $\sub_X(\gamma,\mu_1)=\dbt$ or $\sub_X(\gamma,\mu_2)=\dbt$. 
		Without loss of generality assume that  $\sub_X(\gamma,\mu_1)=\dbt$. In this case  $\gamma[\dbt]\subseteq\mu_1[\dbt]$, but $\gamma[\dbf]\nsubseteq\mu_1[\dbf]$. 
		If $\sub_X(\gamma,\mu_2)\geq\dbf$, then $\sub_X(\gamma,\mu_1)\vee\sub_X(\gamma,\mu_2)=1$, the equality holds trivially.   If $\sub_X(\gamma,\mu_2)\leq\dbt$, then $\gamma[\dbf]\nsubseteq\mu_2[\dbf]$. Since $\gamma[\dbf]$ is an irreducible closed set of $(X,\tau[\dbf])$, it follows that $\gamma[\dbf]\nsubseteq (\mu_1\vee\mu_2)[\dbf]$, and consequently, $$\sub_X(\gamma, \mu_1\vee\mu_2)=\dbt= \sub_X(\gamma,\mu_1)\vee \sub_X(\gamma,\mu_2).$$  
		
		Case 3.   $\sub_X(\gamma,\mu_1)=\dbf$ or $\sub_X(\gamma,\mu_2)=\dbf$.  
		Similar to Case 2. 
		
		Case 4. $\sub_X(\gamma,\mu_1)=0=\sub_X(\gamma,\mu_2)$. 
		We show that $\sub_X(\gamma,\mu_1\vee\mu_2)=0$. Since $\sub_X(\gamma,\mu_i)=0$, then $\gamma[\dbt]\nsubseteq\mu_i[\dbt]$ and  $\gamma[\dbf]\nsubseteq\mu_i[\dbf]$,   $i=1,2$. Since $\gamma[\dbt]$ is an irreducible closed set of $(X,\tau[\dbt])$, then $\gamma[\dbt]\nsubseteq (\mu_1\vee\mu_2)[\dbt]$, hence $\sub_X(\gamma,\mu_1\vee\mu_2)\leq\dbf$.  Likewise,  $\sub_X(\gamma,\mu_1\vee\mu_2)\leq\dbt$. Therefore,  $\sub_X(\gamma,\mu_1\vee\mu_2)=0$.  \end{proof}
	
	Since we do not distinguish $\bbB$-topological spaces with bitopological spaces, we'll say that a bitopologcal space  $(X,\tau[\dbt],\tau[\dbf])$ is  $\bbB$-sober if the space $(X,\tau)$ is $\bbB$-sober.

	We say that a pair $(K_\dbt,K_\dbf)$  of irreducible closed sets of $(X,\tau[\dbt],\tau[\dbf])$ is \emph{represented by} a point $x$ if, $K_\dbt$ is the closure of $\{x\}$ in $(X,\tau[\dbt])$ and $K_\dbf$ is the closure of $\{x\}$ in $(X,\tau[\dbf])$.
	The following proposition provides a characterization of $\bbB$-sobriety and d-sobriety of bitopological spaces in terms of pairs of irreducible closed sets, which is helpful in understanding the difference between them. 
	
	\begin{prop} \label{d-points} Suppose $(X,\tau[\dbt],\tau[\dbf])$ is a bitopological space. \begin{enumerate}[label=\rm(\roman*)] \setlength{\itemsep}{0pt} \item   $(X,\tau[\dbt], \tau[\dbf])$ is $\bbB$-sober if and only if   each pair $(K_\dbt,K_\dbf)$ of irreducible closed sets is represented by a unique point of $X$.  
			\item {\rm (\cite{JM2008})} $(X,\tau[\dbt],\tau[\dbf])$ is d-sober if and only if each pair $(K_\dbt,K_\dbf)$ of irreducible closed sets that satisfies the following conditions is represented by a unique point of $X$:  
			\begin{enumerate}[label=\rm(\alph*)] \setlength{\itemsep}{0pt} 
				\item  For all $U\in\tau[\dbt]$ and $V\in\tau[\dbf]$, if $U\cap V=\emptyset$, then either $K_\dbt\cap U=\emptyset$ or $K_\dbf\cap V=\emptyset$. \item  For all $U\in\tau[\dbt]$ and $V\in\tau[\dbf]$, if $U\cup V=X$, then either $K_\dbt\cap U\not=\emptyset$ or $K_\dbf\cap V\not=\emptyset$. \end{enumerate} \end{enumerate} \end{prop}
	\begin{proof} The verification of (i) is routine; (ii) is   contained  in \cite[Section 4]{JM2008}.  \end{proof}

	\begin{cor}Every $\mathbb{B}$-sober bitopological space is d-sober, hence join sober. \end{cor}
	
	But, a d-sober bitopological space need not be $\bbB$-sober, the space $(\mathbb{R},\tau)$ in Example \ref{eg2} provides a counterexample. First, since $X$ is order-separated, it is d-sober by \cite[Theorem 3.9]{JM2008}. Second, since the pair of irreducible closed sets $([1,\infty),(-\infty,0])$ is not represented by any point, it is not $\bbB$-sober.

	\begin{cor}\label{sobrification} Let $(X,\tau)$ be a $\bbB$-topological space. Then $(X,\tau)$ is d-sober if and only if each irreducible closed set $\gamma$ of $(X,\tau)$ that satisfies the following conditions is the closure of $1_x$ for a unique point $x$ of $X$: \begin{enumerate}[label=\rm(\alph*)] \setlength{\itemsep}{0pt} 
			\item  For each closed set $\mu$ of $(X,\tau)$, if $\mu(z)\not=0$ for all $z\in X$, then either $\dbt\wedge\gamma\leq\mu$ or $\dbf\wedge\gamma\leq\mu$. \item For each closed set $\mu$ of $(X,\tau)$, if $\mu(z)\not=1$ for all $z\in X$, then  $\gamma\not\leq\mu$. \end{enumerate}  
	\end{cor}

	By  \cite[Theorem 3.7]{Zhang2018} the $\bbB$-sobrification of $(X,\tau)$ is given by the space $(\hat{X},\hat{\tau})$, where $\hat{X}$ is the set of irreducible closed sets of $(X,\tau)$, and $\hat{\tau}$ is the collection of $\bbB$-valued subsets $$\hat{\lam}\colon \hat{X}\lra\bbB,\quad \gamma\mapsto \neg\sub_X(\gamma,\neg\lam),  
	\quad \lam\in\tau. $$ This construction is an extension of that of sobrification of topological spaces in terms of closed sets (see e.g. \cite{Goubault}) to the $\bbB$-valued context. 
	The subspace of $(\hat{X},\hat{\tau})$ composed of irreducible closed sets that satisfy the requirements  (a) and (b) in Corollary \ref{sobrification}  is the d-sobrification of $(X,\tau)$.

	\begin{exmp} $\bbB$-sobriety and componentwise sobriety do not imply each other.    
		\begin{enumerate}[label=\rm(\roman*)] \setlength{\itemsep}{0pt} \item  Given topological spaces $X$ and $Y$, consider the $\bbB$-topological space $X.Y= (X\times Y,\tau)$ in Example \ref{X.Y}.   It is not hard to check that each irreducible closed set of $(X\times Y,\tau[\dbt])$ is of the form $H \times Y$ with $H$ being an irreducible closed set of $X$;  each irreducible closed set  of $(X\times Y,\tau[\dbf])$ is of the form $X \times J$ with $J$ being an irreducible closed set of $Y$. With help of this fact one verifies that $X.Y$ is  $\bbB$-sober if and only if both of the topological  spaces $X$ and $Y$ are sober. But, the space  $X.Y$ is not componentwise $T_0$ unless both $X$ and $Y$ are singleton space. So, $\bbB$-sober $\bbB$-topological spaces need not be componentwise sober.  
			
			\item Let $X$ be a sober topological space. It is readily verified that the componentwise Hausdorff (hence Hausdorff and componentwise sober) space $\omega(X)$ is not $\bbB$-sober unless $X$ is a singleton space. Furthermore,  the $\bbB$-sobrification of   $\omega(X)$ is the space $X.X$.  \end{enumerate}  \end{exmp}

	\subsection{A Hofmann-Mislove theorem}
	
	The Hofmann-Mislove theorem \cite{HM1981,Gierz2003} says that for a sober space $X$, the poset of compact saturated subsets of $X$ (ordered by reverse inclusion) is isomorphic to the poset of Scott open filters of the open-set lattice  of $X$ (ordered by inclusion). In this subsection we establish a similar result for $\bbB$-sober spaces. In \cite{JM2006,JM2008}, making use of the d-frame structure of open sets of a bitopological space, Jung and Moshier have established a Hofmann-Mislove theorem for bitopological spaces, see \cite[Theorem 6.6]{JM2008}. The Hofmann-Mislove theorem established in this paper,  see Theorem \ref{HMT for Bitop} below,  only makes use of the fact that the open sets of a bitopological space form an object of the slice category $\BFrm$. While the result of Jung and Moshier establishes a connection between bitopological spaces and d-frames, the result presented here is about the connection between bitopological spaces and objects of a slice category of frames. So, the approaches  and the results are quite different.
	
	We need counterpart notions of compact sets, saturated sets, and Scott open filters in the bitopological context.
	
	\begin{defn} Suppose $(X,\tau)$ is a $\bbB$-topological space and   $\theta\in\bbB^X$. We say that $\theta$ is a compact ($\bbB$-valued) set of $(X,\tau)$ if for each directed family $\Lambda$ (w.r.t pointwise order) of open sets of $(X,\tau)$,  
		$$\sub_X\Big(\theta,\bigvee\Lambda\Big)= \bigvee_{\lam\in\Lambda}\sub_X(\theta,\lam).$$ 
	\end{defn}
	
	\begin{prop}\label{compact=pairwise}  Suppose $(X,\tau)$ is a $\bbB$-topological space and   $\theta\in\bbB^X$. Then  $\theta$ is   compact in $(X,\tau)$ if and only if $\theta[\dbt]$ is   compact in the topological space $(X,\tau[\dbt])$ and $\theta[\dbf]$ is a compact in the topological space $(X,\tau[\dbf])$. \end{prop}
	\begin{proof}This follows from the fact that for each family $\Lambda$ of open sets of $(X,\tau)$, $$\dbt\leq\sub_X\Big(\theta,\bigvee\Lambda\Big)\iff \theta[\dbt]\subseteq \bigcup\{\lam[\dbt]: \lam\in\Lambda\}$$ and that \begin{align*}\dbf\leq\sub_X\Big(\theta,\bigvee\Lambda\Big)&\iff \theta[\dbf]\subseteq \bigcup\{\lam[\dbf]: \lam\in\Lambda\}. \qedhere\end{align*} \end{proof}
	
	\begin{cor}  Suppose $(X,\tau)$ is a $\bbB$-topological space and   $\theta\in\bbB^X$. Then   $\theta$ is compact in $(X,\tau)$ if and only if for each directed family $\Lambda$ of open sets, $\theta\leq\bigvee\Lambda$ implies that $\theta\leq \lam$  for some $\lam\in \Lambda$. \end{cor}
	  
	A $\bbB$-topological space $(X,\tau)$ is  \emph{compact} if the constant function $X\lra\bbB$ with value $1$ is compact. It follows from Proposition \ref{compact=pairwise} that $(X,\tau)$ is  compact if and only if it is componentwise compact; that is, both of  topological spaces  $(X,\tau[\dbt])$ and $(X,\tau[\dbf])$ are compact. 
	
	\begin{prop}Every Hausdorff and compact $\bbB$-topological space is normal; every regular and compact $\bbB$-topological space is normal. \end{prop}
	
	\begin{proof} This follows from Proposition \ref{normal=binormal} and the fact that a compact $R_1$ topological space is normal \cite[Theorem 8]{R1space}. \end{proof}
	
	\begin{defn}Suppose $(X,\tau)$ is a $\bbB$-topological space and    $\theta\in\bbB^X$. We say that $\theta$ is a saturated ($\bbB$-valued) set of $(X,\tau)$ if the map $\theta\colon(X,\Omega(\tau))\lra(\bbB,\ra)$ preserves $\bbB$-order; i.e.,  $\Omega(\tau)(x,y)\leq\theta(x)\ra \theta(y)$ for all $x,y\in X$. \end{defn}
	
	The (pointwise) join and meet of any family of saturated  sets of a $\bbB$-topological space are saturated. 
	
	\begin{prop}\label{saturated set as intersection} Suppose $(X,\tau)$ is a $\bbB$-topological space and $\theta\in\bbB^X$. Then   $\theta$ is a saturated set of $(X,\tau)$ if and only if $\theta$ is the intersection of a family of open sets of $(X,\tau)$. \end{prop} 
	
	\begin{proof}By definition we have $$\Omega(\tau)(x,y)= \bigwedge_{\lam\in\tau}\lam(x)\ra\lam(y), $$ so every open set is  saturated, hence the intersection of any family of open sets is saturated. This proves the sufficiency. 
		
		For necessity, suppose $\theta\colon X\lra\bbB$ is a  saturated set of $(X,\tau)$. Then for each $x$ and each $y$, $$\theta(x)\leq\Omega(\tau)(x,y)\ra\theta(y)= \overline{1_y}(x)\ra\theta(y) = (\overline{1_y}\ra\theta(y))(x).$$ Since $\overline{1_y}\ra\theta(y)=\theta(y)_X\vee \neg(\overline{1_y})$, then $\overline{1_y}\ra\theta(y)$ is open and $\theta$ is contained in the intersection of the family of open sets $\{\overline{1_y}\ra\theta(y): y\in X\}.$  Finally, since $\theta(x)=\overline{1_x}(x)\ra\theta(x)$ for all $x\in X$,  $\theta$ equals the intersection of that family of open sets, as desired. \end{proof}
	
	Before presenting the next definition, we recall that for each $\bbB$-topological space $(X,\tau)$, the map $i_X\colon\bbB\lra\tau$ assigns to each element $b$ of $\bbB$ the constant open set $X\lra\bbB$ with value $b$; i.e., $i_X(b)=b_X$. 
	\begin{defn} Suppose $(X,\tau)$ is a $\bbB$-topological space. A  map $F\colon\tau\lra\bbB$ is a $\bbB$-filter of $\tau$ if it satisfies  \begin{enumerate}[label=\rm(F\arabic*)] \setlength{\itemsep}{0pt} \item $F\circ i_X=\id_\bbB$, i.e., $F(b_X)=b$ for all $b\in\bbB$; \item $F(\lam\wedge\mu)= F(\lam)\wedge F(\mu)$ for all $\lam,\mu\in\tau$.  
		\end{enumerate}
		
		A $\bbB$-filter is Scott open if it satisfies  moreover \begin{enumerate}\item[\rm (F3)] $F\big(\bigvee\Lambda\big)= \bigvee_{\lam\in \Lambda}F(\lam)$ for each directed subset $\Lambda$ of $\tau$.
	\end{enumerate} \end{defn}
	
	For each $\bbB$-filter $F$ and each element $\lam$ of $\tau$, from (F1) and (F2) one verifies  that $F(\dbt\wedge\lam) =\dbt\wedge F(\lam)$ and $F(\dbf\wedge\lam) =\dbf\wedge F(\lam)$, hence $F(\lam)=F(\dbt\wedge\lam)\vee F(\dbf\wedge\lam)$. 
	
	With  the notions of compact sets, saturated sets, and Scott open $\bbB$-filters at hand, we are able to formulate and prove a Hofmann-Mislove theorem for bitopological spaces. We begin with a lemma.
	
	\begin{lem}\label{compact sets generate scott open} Suppose $(X,\tau)$ is a $\bbB$-topological space  and   $\theta\in\bbB^X$.   \begin{enumerate}[label=\rm(\roman*)] \setlength{\itemsep}{0pt} \item The map $\sub_X(\theta,-)\colon\tau\lra\bbB$ is a $\bbB$-filter if and only if $\theta$ is inhabited; that is, $\bigvee_{x\in X}\theta(x)=1$. \item  The map $\sub_X(\theta,-)\colon\tau\lra\bbB$ is a Scott open $\bbB$-filter if and only if $\theta$ is inhabited and compact.\end{enumerate}\end{lem}
	
	\begin{proof} (i) If $\sub_X(\theta,-)\colon \tau\lra\bbB$ is a $\bbB$-filter, then by (F1) we have $$0=\sub_X(\theta,0_X)= \bigwedge_{x\in X} (\theta(x)\ra 0) =  \Big(\bigvee_{x\in X}\theta(x)\Big)\ra 0,$$ which shows that $\theta$ is inhabited. Conversely,  if $\theta$ is inhabited, we show that $\sub_X(\theta,-)$ is a $\bbB$-filter of $\tau$. It is trivial that $\sub_X(\theta,-)$ satisfies (F2). For (F1)  we calculate: for each $b\in\bbB$, $$\sub_X(\theta,b_X)= \bigwedge_{x\in X} (\theta(x)\ra b) =  \Big(\bigvee_{x\in X}\theta(x)\Big)\ra b= b.$$  
		
		(ii) This follows from a combination of the following facts:   $\sub_X(\theta,-)$ satisfies (F1)   if and only if  $\theta$ is inhabited; $\sub_X(\theta,-)$ satisfies (F3) if and only if  $\theta$ is compact;   $\sub_X(\theta,-)$ always satisfies (F2).  \end{proof} 
	
	\begin{thm}\label{HMT for Bitop} Suppose $(X,\tau)$ is a $\bbB$-topological space.  \begin{enumerate}[label=\rm(\roman*)] \setlength{\itemsep}{0pt} \item  For each inhabited, saturated and compact set $\theta$ of $(X,\tau)$, the map $\sub_X(\theta,-)\colon \tau\lra\bbB$ is a Scott open $\bbB$-filter of $\tau$. \item For all inhabited, saturated and compact  sets $\theta_1$ and $\theta_2$ of $(X,\tau)$, it holds that $$\sub_X(\theta_2,\theta_1)=\bw_{\lam\in\tau}(\sub_X(\theta_1,\lam)\ra\sub_X(\theta_2,\lam)).$$  While the left side measures the truth degree that $\theta_2$ is contained in $\theta_1$, the right side measures  the truth degree that   $\sub_X(\theta_1,-)$ is contained in  $\sub_X(\theta_2,-)$.
			\item If $(X,\tau)$ is $\bbB$-sober, then for each Scott open $\bbB$-filter $F$ of $\tau$, there is a unique inhabited, saturated and compact  set $\theta$ of $(X,\tau)$ such that $F =\sub_X(\theta,-)$. \end{enumerate} \end{thm}
	
	\begin{proof}(i) Lemma \ref{compact sets generate scott open}. 
		
		(ii) First we show that if $ \sub_X(\theta_2,\theta_1)\leq\dbt$  then $$\bw_{\lam\in\tau}(\sub_X(\theta_1,\lam)\ra\sub_X(\theta_2,\lam))\leq \dbt. $$   Since $ \sub_X(\theta_2,\theta_1)\leq\dbt$, then $\theta_2(x)\ra\theta_1(x)  \leq\dbt$ for some $x\in X$. Consider the open set $\mu$ of $(X,\tau)$ given by $$\mu(y)=\overline{1_x}(y)\ra\theta_1(x)=\neg(\overline{1_x}(y))\vee\theta_1(x) .$$ Since $\theta_1$ is saturated, it follows that   $\overline{1_x}(y)=\Omega(\tau)(y,x)\leq \theta_1(y)\ra\theta_1(x)$, then $$\theta_1(y) \leq \overline{1_x}(y)\ra\theta_1(x)=\mu(y),$$    which shows that $\sub_X(\theta_1,\mu)=1$. Since    $$\sub_X(\theta_2,\mu)\leq \theta_2(x)\ra \mu(x) = \theta_2(x)\ra\theta_1(x) \leq\dbt,$$  then $$\bw_{\lam\in\tau}(\sub_X(\theta_1,\lam)\ra\sub_X(\theta_2,\lam))\leq \sub_X(\theta_1,\mu)\ra\sub_X(\theta_2,\mu)\leq \dbt. $$ Likewise,  if $ \sub_X(\theta_2,\theta_1)\leq\dbf$  then $$\bw_{\lam\in\tau}(\sub_X(\theta_1,\lam)\ra\sub_X(\theta_2,\lam))\leq \dbf. $$ Therefore, $$\bw_{\lam\in\tau}(\sub_X(\theta_1,\lam)\ra\sub_X(\theta_2,\lam))\leq \sub_X(\theta_2,\theta_1). $$ 
		
		The converse inequality is trivial. 
		
		(iii) Suppose $(X,\tau)$ is $\bbB$-sober and $F\colon \tau\lra\bbB$ is a Scott open $\bbB$-filter of $\tau$.  We wish to prove that there is a unique inhabited, saturated and compact  set $\theta$ of $(X,\tau)$ such that $F(\lam)=\sub_X(\theta,\lam)$ for all $\lam\in\tau$. The proof is an adaptation  of that for the Hofmann-Mislove theorem  in \cite{KP1994} to the bitopological context. 
		
		We prove the existence of $\theta$ first. 
		Let $$\CF =\{U\in\tau[\dbt]:   F(\dbt_U)=\dbt\} \quad\text{and}\quad \CG=\{V\in\tau[\dbf]:  F(\dbf_V)=\dbf\}.$$ Making use of the equality $F(\dbt\wedge\lam) =\dbt\wedge F(\lam)$  one readily verifies that $\CF$ is a Scott open filter of the topology $\tau[\dbt]$. Likewise, $\CG$ is a Scott open filter of the topology $\tau[\dbf]$.   Let $S_\dbt$ be the intersection of all members of $\CF$, $S_\dbf$ be the intersection of all members of $\CG$. 
		Put $$\theta_\dbt= \dbt_{S_\dbt}, \quad \theta_\dbf=\dbf_{S_\dbf},\quad  \theta=\theta_\dbt\vee\theta_\dbf.$$  
		We show that $\theta$ satisfies the requirement. 
		Since both $\theta_\dbt$ and $\theta_\dbf$ are saturated, $\theta$ is saturated. If we could show that $F(\lam)=\sub_X(\theta,\lam)$ for all $\lam\in \tau$, then  $\theta$ would be compact and inhabited by Lemma \ref{compact sets generate scott open}. We do this in two steps.
		
		{\bf Step 1}.   $F(\lam)\leq \sub_X(\theta,\lam)$. If $\dbt\leq F(\lam)$, then   $\theta_\dbt\leq \lam$, hence $$\sub_X(\theta,\lam) =\sub_X(\theta_\dbt,\lam)\wedge\sub_X(\theta_\dbf,\lam)\geq 1\wedge\dbt=\dbt.$$ Likewise, if $\dbf\leq F(\lam)$, then $\dbf\leq\sub_X(\theta,\lam)$. Therefore, $F(\lam)\leq \sub_X(\theta,\lam)$. 
		
		{\bf Step 2}.   $F(\lam)\geq \sub_X(\theta,\lam)$. We distinguish three cases.  
		
		Case 1. $\sub_X(\theta,\lam)=1$. By definition of $\sub_X$ we have $$S_\dbt=\theta[\dbt]\subseteq \lam[\dbt] \quad\text{and}\quad S_\dbf=\theta[\dbf]\subseteq \lam[\dbf].$$ Suppose on the contrary that $F(\lam)\not=1$. From $F(\lam)=F(\dbt\wedge\lam)\vee F(\dbf\wedge\lam)$ one deduces that either $\lam[\dbt]\notin\CF$ or $\lam[\dbf]\notin\CG$.  Without loss of generality we assume that $\lam[\dbt]$  is not in $\CF$.  Take an open set $U_0$ of $(X,\tau[\dbt])$ that contains $\lam[\dbt]$ and is maximal with respect to not being in the filter $\CF$; take an open set $V_0$ of $(X,\tau[\dbf])$ that is maximal with respect to not being in the filter $\CG$. Let $K_\dbt$ be the complement of $U_0$ and let $K_\dbf$ be the complement of $V_0$. From maximality of $U_0$ and $V_0$ one verifies that $(K_\dbt,K_\dbf)$ is a pair of irreducible closed sets of the bitopological space $(X,\tau[\dbt],\tau[\dbf])$.  Since $(X,\tau)$ is $\bbB$-sober, there is a unique point $x$ of $X$ such that $K_\dbt$ is the closure of $\{x\}$ in $(X,\tau[\dbt])$ and $K_\dbf$ is the closure of $\{x\}$ in $(X,\tau[\dbf])$. We claim that every member $U$ of $\CF$ contains $x$, for otherwise $U$ misses the closure of $\{x\}$ in $(X,\tau[\dbt])$, and
		hence $U \subseteq U_0$, which would imply $U_0\in\CF$, a contradiction. Therefore, $x\in S_\dbt\subseteq \lam[\dbt]$,   contradicting that $\lam[\dbt]\subseteq U_0$.  
		
		Case 2. $\sub_X(\theta,\lam)=\dbt$. By definition of $\sub_X$ we have $\theta\leq\dbt\ra\lam$, then $F(\dbt\ra\lam)=1$ by Case 1. Therefore $$ F(\lam)\geq F(\dbt\wedge(\dbt\ra\lam))=\dbt\wedge 1=\dbt.$$
		
		Case 3. $\sub_X(\theta,\lam)=\dbf$. Similar to Case 2.  
		
		Now we prove the uniqueness of $\theta$. Suppose $\theta\in\bbB^X$ satisfies the requirement. Since $\theta$ is saturated,  it is an intersection of a family of open sets by Proposition \ref{saturated set as intersection}. Since  $F(\lam)=\sub_X(\theta,\lam)$  and $$\sub_X(\theta,\lam)=1\iff\theta\leq\lam$$ for each open set $\lam$,  it follows that $\theta=\bw\{\lam\in\tau:F(\lam)=1\}$, which proves the uniqueness of $\theta$. \end{proof}

\end{document}